\documentclass[11pt]{amsart}
\usepackage{amsbsy,amssymb,amscd,amsfonts,latexsym,amstext,delarray, amsmath,graphicx,color,
amsthm, enumerate, verbatim}
\usepackage{amstext,amsopn, mathtools, bbm}
\usepackage{graphicx}
\usepackage{mathtools}
\usepackage{hyperref}
\usepackage{enumitem}
\usepackage{epsfig}
\input xy
\usepackage[letterpaper, margin=1.5in]{geometry}

\newcommand{\dd}{\text{d}}

\def\sm{{\mathcal M}}

\newcommand\R{{\mathbb R}}

\hypersetup{
	colorlinks=true,
	linkcolor=blue,
	filecolor=blue,      
	urlcolor=blue,
	citecolor = blue
}

	\DeclareMathOperator{\supp}{supp}
	\DeclareMathOperator{\spn}{span}
	
	\DeclareMathOperator{\proj}{proj}
	\DeclareMathOperator{\ran}{ran}

	\newtheorem{lemma}{Lemma}[section]
	
	\newtheorem{prop}[lemma]{Proposition}
	\newtheorem{cor}{Corollary}[section]
	
	\newtheorem{theorem}{Theorem}

	\theoremstyle{remark}
	\newtheorem{remark}{Remark}[section]
	
\DeclareFontFamily{U}{mathx}{\hyphenchar\font45}
\DeclareFontShape{U}{mathx}{m}{n}{
      <5> <6> <7> <8> <9> <10>
      <10.95> <12> <14.4> <17.28> <20.74> <24.88>
      mathx10
      }{}
\DeclareSymbolFont{mathx}{U}{mathx}{m}{n}
\DeclareFontSubstitution{U}{mathx}{m}{n}
\DeclareMathAccent{\widecheck}{0}{mathx}{"71}
\DeclareMathAccent{\wideparen}{0}{mathx}{"75}

\begin{document}
\title{Fourier decay of fractal measures on hyperboloids}

\author[Barron]{Alex Barron}
\address{Department of Mathematics, University of Illinois, Urbana, IL 61801, U.S.A.}
\email{aabarron@illinois.edu} 

\author[Erdo\u{g}an]{M. Burak Erdo\u{g}an}
\address{Department of Mathematics, University of Illinois, Urbana, IL 61801, U.S.A.}
\email{berdogan@illinois.edu}

\author[Harris]{Terence L.~J.~Harris}
\address{Department of Mathematics, University of Illinois, Urbana, IL 61801, U.S.A.}
\email{terence2@illinois.edu}

\maketitle

\begin{abstract}
	Let $\mu$ be an $\alpha$-dimensional probability measure. We prove new upper and lower bounds on the decay rate of hyperbolic averages of the Fourier transform $\widehat{\mu}$. More precisely, if $\mathbb{H}$ is a truncated hyperbolic paraboloid in $\R^d$ we study the optimal $\beta$ for which $$\int_{\mathbb{H}} |\widehat{\mu}(R\xi)|^2  \, d  \sigma (\xi)\leq C(\alpha, \mu) R^{-\beta}$$ for all $R > 1$. Our estimates for $\beta$ depend on the minimum between the number of positive and negative principal curvatures of $\mathbb{H}$; if this number is as large as possible our estimates are sharp in all dimensions. 
\end{abstract}

\section{Introduction}

 \footnote{2010 \textit{Mathematics Subject Classification}. Primary 42B37
}
Let $\sm_d$ be the space of  non-negative finite Borel measures supported in
 $B(0,1)\subset\R^d$. For $\alpha\in (0,d)$, the  $\alpha$-dimensional energy of $\mu\in\sm_d$ is defined via
\begin{equation}
I_\alpha(\mu):=\iint \frac{\dd\mu(x)\, \dd\mu(y)}{|x-y|^{\alpha}}=
c_{\alpha,d} \int\frac{|\widehat{\mu}(\xi)|^2}{|\xi|^{d-\alpha}} \, \dd\xi\nonumber,
\end{equation}
where $\widehat{\mu}$ is the Fourier transform of the measure $\mu$:
$$
\widehat{\mu}(\xi)=\int e^{-ix\cdot\xi} \, \dd\mu(x).
$$

We are interested in the decay of $\widehat{\mu}(\xi)$ at infinity for measures $\mu$ with finite energy.
Although  $I_\alpha(\mu)<\infty$ does not imply any pointwise decay of
$|\widehat{\mu}(\xi)|$ as $|\xi| \rightarrow \infty$, in general, the averages of $\widehat{\mu}(\xi)$ decay at infinity.
Let $\Gamma$ be a smooth, compact submanifold of $\R^d$. Let 
$\sigma$ be the surface measure on $\Gamma$. 
The following can be considered as a variant of the Fourier restriction problem:  
 
Fix $\alpha \in (0,d)$. For which $\beta>0$ is
\begin{equation}\label{gammaavg}
\int_\Gamma |\widehat{\mu}(R\xi)|^2  \, d  \sigma (\xi)\leq C_\beta R^{-\beta} I_\alpha(\mu), 
\end{equation}
for all $R>1$?  Let $\beta(\alpha,\Gamma)$ denote the supremum of all $\beta$ so that \eqref{gammaavg} holds for all $\mu\in \sm_d$. 
 
This question was first formulated by Mattila for $\Gamma=S^{d-1}$, \cite{mat}, in his  work on Falconer's distance set problem, \cite{fal}, and intersection theory of general sets. When  $\Gamma= S^{d-1}$, the bound \eqref{gammaavg} with $\beta=d-\alpha$ implies that compact sets in $\R^d$ with Hausdorff dimension greater than  $ \alpha$ have positive measure distance sets, \cite{mat}. There are also more recent, improved, applications  to the distance set problem, see \cite{GIOW}, and  to the pinned distance set problem, see \cite{Liu1,Liu2}. See \cite{HI,IL} for applications of averages on elliptic surfaces to the distance set problem with respect to more general metrics.   There are further applications    to the upper bounds on the dimension of the sets on which convergence to the initial data for the Schr\"odinger equation fails, \cite{bbcr,keith}, and to  dispersive estimates for the linear Schr\"odinger evolution associated to an operator $-\Delta+\mu$ in $\mathbb R^d$, $d\geq 3$,   where the potential $\mu$ is a signed measure with sufficiently large fractal 
dimension, \cite{goldmeasure,egg}. 
Recall that the free Schr\"odinger resolvent $R_0(\lambda^2\pm i0)$, $\lambda>0$,  acts by multiplying Fourier transforms
pointwise by the distribution
\begin{equation*}
\text{p.v.}\frac{1}{|\xi|^2 - \lambda^2} 
\pm i\frac{\pi}{\lambda}\,d\sigma_{\lambda S^{d-1}}.
\end{equation*}
on the Fourier side. Therefore, the spherical averages lead  to a uniform in $\lambda$ estimate  for $\|(I + R_0^+(\lambda^2)\mu)^{-1}\|_{L^2_\mu\to L^2_\mu}$, see \cite{egg}, which is crucial in the study of dispersive decay estimates.

When $\Gamma= S^{1}$, the sharp range of $\beta$  was obtained by Wolff in \cite{wolff1}, also see  \cite{mat,sjo,E}. The best known results in higher dimensions are due to Du and Zhang, \cite{DZ}. For other partial results and counterexamples for $\Gamma=S^{d-1}$ or a codimension 1 manifold with positive principal curvatures, see \cite{mat,sjo,bou1,wolff1,mat2,E,E2,bbcrv,keith,DGOWWZ,DZ,D}.   Also see \cite{edo,ChHaLe2,HL} for results when $\Gamma$ is a curve.

The case when $\Gamma$ is  the truncated light cone,  $\Gamma=\{(x,t)\in\R^d\times \R: |x|=t\in[1,2]\}$, was studied in \cite{wolff3,E,ChHaLe1,Ob, H, TH1}. Optimal results are known in dimensions $d=2,3$, see \cite{E, ChHaLe1}.  Also see, \cite[Eq.~3.27]{TH2} for the best known estimates in dimensions $d\geq 4$.
These estimates imply fractal Strichartz inequalities for the wave equation, see \cite{wolff3} (p.1283-1287) and \cite{E,ChHaLe1,R,H,TH1}.  The conic case is also useful for projection theorems.  For example, using conic averages Oberlin and Oberlin, \cite{ObOb}, studied a version of Marstrand's projection theorem  in $\R^3$ concerning the Hausdorff dimension of  projections onto a restricted family of planes defined  by a curve in the cone. This application was further developed in \cite{TH2}. Another interesting application of the conic case was proposed in \cite{R}; the fractal Strichartz inequalities for measures which are the tensor product of an $\alpha$ dimensional measure in $\R^d$ and the Lebesgue measure in $[0,1]$ also imply lower bounds for the Hausdorff dimension of distance sets. 

 In this paper we study  the case when $\Gamma$ is an hyperboloid. 
 Let $M$ be a diagonal $(d-1) \times (d-1)$ matrix with all nonzero entries equal to $\pm 1$ (with at least two entries having opposite sign) and let $\mathbb{H}$ 
  be the surface $$\mathbb{H} = \left\{ (\xi,  \langle M \xi, \xi \rangle ) : \xi \in B^{d-1}(0,1) \right\}.$$ 
Let $p_M$ denote the number of positive entries in $M$ and $n_M$ the number of negative entries in $M$, and let $$m = \min(p_M, n_M).$$  The ordering of signs is unimportant, so assume without loss of generality that along the diagonal $M$ has $(d-1-m)$ positive signs followed by $m$ negative signs, and write $\mathbb{H}^{d-1}_m = \mathbb{H}$. The quantity $m$ is related to the \textit{signature} $s$ of the matrix by the formula $m = \frac{d-1-s}{2}$. 

Finally let $\sigma$ be the associated surface measure.  
Below we prove the following.
 
 \begin{theorem}\label{thm:knownBounds} Let $d \geq 3$. If $d- m - 1 \leq \alpha < d-m$ then 
\begin{equation} \label{upperbound} \beta(\alpha, \mathbb{H}^{d-1}_m)  \leq \alpha - \frac{\alpha - m}{d-2m}. \end{equation}
On the other hand we have, 
\[ \beta(\alpha, \mathbb{H}^{d-1}_1) \geq \frac{\alpha(d-2)}{d-1}, \quad \alpha \in [0, d-1), \]
and for $m > 1$ and $j \in [1,m-1]$, 
\[ \beta(\alpha, \mathbb{H}^{d-1}_m) \geq	\begin{cases} \frac{\alpha(d-j-1)}{d-j}, &\alpha \in \left[ \frac{(j-1)(d-j)}{m-1}, \frac{j(d-j)}{m} \right) \\
\alpha - \frac{\alpha-j}{d-j-m}, &\alpha \in \left[ \frac{j(d-j)}{m}, \frac{j(d-(j+1))}{m-1} \right).\end{cases} \] 
If $m < \frac{d-1}{2}$ then 
\begin{equation} \label{bilinear} \beta(\alpha, \mathbb{H}^{d-1}_m) \geq \frac{\alpha}{2} + \frac{d}{4} - \frac{1}{2}, \quad \alpha \in \left(\frac{d-1}{2}  ,\frac{d}{2} + 1\right). \end{equation}
Additionally, for any $m$ and $d$, 
	 \begin{equation}\label{eq:elementary}
	\begin{aligned} 
\beta(\alpha, \mathbb{H}^{d-1}_m) &= \alpha,\qquad && \alpha < \frac{d-1}{2} \\
\beta(\alpha, \mathbb{H}^{d-1}_m) &\geq \frac{d-1}{2}, \qquad \frac{d-1}{2} \leq &&\alpha \leq \frac{d+1}{2}  \\
\beta(\alpha, \mathbb{H}^{d-1}_m) &\geq \alpha-1,\qquad \frac{d+1}{2} \leq &&\alpha \leq d-m \\
\beta(\alpha, \mathbb{H}^{d-1}_m) &= \alpha-1,\qquad &&\alpha \geq d-m, \end{aligned}
	\end{equation} (note that we always have $m \leq \frac{d-1}{2},$ and therefore $d-m \geq \frac{d+1}{2}$).\end{theorem}

\noindent This shows that the case $\alpha \geq d-m$ is completely understood. We also have the following corollary. 
\begin{cor} If $d \geq 3$ is odd and $m = \frac{d-1}{2}$, then
\[ \beta(\alpha, \mathbb{H}^{d-1}_m) =  \begin{cases} \alpha &\text{ if } \alpha \in \left( 0, \frac{d-1}{2} \right) \\
\frac{d-1}{2} &\text{ if } \alpha \in \left[ \frac{d-1}{2}, \frac{d+1}{2} \right] \\
\alpha-1 &\text{ if } \alpha \in \left( \frac{d+1}{2}, d \right]. \end{cases} \]
If $d \geq 4$ is even and $m = \frac{d}{2} - 1$, then 
\[ \beta(\alpha, \mathbb{H}^{d-1}_m) =  \begin{cases} \alpha &\text{ if } \alpha \in \left( 0, \frac{d-1}{2} \right) \\
\frac{d-1}{2} &\text{ if } \alpha \in \left[ \frac{d-1}{2}, \frac{d}{2} \right] \\
\frac{\alpha}{2} + \frac{d}{4} - \frac{1}{2} &\text{ if } \alpha \in \left[ \frac{d}{2}, \frac{d}{2}+1 \right] \\
\alpha-1 &\text{ if } \alpha \in \left( \frac{d}{2}+1, d \right]. \end{cases} \]
\end{cor} 

\noindent In particular note the decay rate is completely determined in the case $d =3$ and $d = 4$ (since we must have $m = 1$). When $d$ is odd this corollary follows from the fact that $\frac{d-1}{2} = d - m -1$ precisely when $m = \frac{d-1}{2}$. Note also that in this case $d - m = \frac{d+1}{2}$. When $d$ is even the corollary follows from \eqref{upperbound}, \eqref{bilinear} and \eqref{eq:elementary}. The second equality also holds when $d=2$, but this reduces to Wolff's theorem for the parabolic or circular averages. 

 As in the elliptic case there is an equivalent formulation of our main problem in terms of weighted restriction estimates for the extension operator $Ef$ associated to $\mathbb{H}_{m}^{d-1}$. In particular if $\mu_{R}$ is an $\alpha$-dimensional measure supported in $B^{d}(0,R)$ then to prove Theorem \ref{thm:knownBounds} it suffices to study optimal $s(\alpha)$ for which $$\frac{\|Ef\|_{L^{2}(\mu_{R})}}{\|f\|_{L^2}} \lesssim R^{s(\alpha)}, \ \ \ \ \ \text{supp}(f) \subset B^{d-1}(0,2). $$ We discuss how to make this dependence precise below in Section \ref{sec:equiv}. After discretizing the measure we can reduce matters to studying estimates of the form \begin{equation}\label{eq:weighted0} \|Ef\|_{L^{2} (X)} \lesssim R^{s(\alpha)}\|f\|_{L^2} \end{equation} where $X$ is an $\alpha$-dimensional union of lattice unit cubes contained in $B^{d}(0,R)$ (see the beginning of Section \ref{three} for the precise definition). These estimates can be studied using recent techniques established by Du and Zhang in \cite{DZ} to study weighted estimates for paraboloids. However new ideas are needed to prove the full range of bounds in Theorem \ref{thm:knownBounds} since a direct application of the Du-Zhang argument to $\mathbb{H}_{m}^{d-1}$ yields sub-optimal results in a range of cases. We briefly give an example that helps explain why this is the case. 
 
 An argument due to Rogers, Vargas, and Vega \cite{RVV} implies that there is a $(d-1)$-dimensional measure $\nu_R$ and a function $f \in L^2$ supported in the unit ball such that $$ \|Ef\|_{L^2 (\nu_R)} \gtrsim R^{\frac{1}{2}} \|f\|_{L^2}.$$ Since one always has $$\|Ef\|_{L^{2}(\nu_R)} \lesssim R^{\frac{1}{2}} \|f\|_{L^{2}}$$ this immediately implies that the optimal $s(\alpha)$ in \eqref{eq:weighted0} is $s(\alpha) = \frac{1}{2}$ when $\alpha \geq d-1$. We will show below in Section \ref{three} that in fact the optimal value is $s(\alpha) = \frac{1}{2}$ for all $\alpha \geq d-m$. Our example is similar to the example found in \cite{RVV}. This contrasts the case of the paraboloid where the case $\alpha = d-1$ is much more difficult and was only recently understood in the case $d = 3$ by Du, Guth, and Li, \cite{DGL}, and in general dimensions by Du and Zhang, \cite{DZ}. Moreover, if $m \geq 1$ we see that we have a large `trivial' range where $\alpha \in [d-m, d],$ and thus we need to focus on the case of smaller $\alpha$ with $\alpha < d-m$. A direct application of the Du-Zhang method does not give optimal results in this range for any $\alpha$, and so we need to augment their approach with some new ideas adapted to the geometry of $\mathbb{H}_{m}^{d-1}$.
 
 After discussing some counterexamples in Section \ref{three} we explain in Section \ref{four} how to modify the Du-Zhang method to obtain better bounds on $\beta(\alpha, \mathbb{H}_{m}^{d-1})$ when $\alpha < d-m$ and $m \geq 1$. The general scheme of the argument is the same as in \cite{DZ}, although we need to optimize between different levels in the `broad' and `narrow' cases depending on $\alpha$ and $m$. In order to make this argument work we need to use $k$-narrow decoupling estimates for $\mathbb{H}_{m}^{d-1}$ which follow from arguments due to Bourgain and Demeter in \cite{BD2}. These are discussed below in Section \ref{sec:narrow}. 
 
 However, this is not sufficient to prove the lower bounds claimed in \eqref{bilinear} in Theorem \ref{thm:knownBounds} (and in particular not sufficient to obtain optimal bounds in the range $\alpha \in [2,3)$ when $d = 4$). These bounds require a bilinear argument that invokes weighted versions of bilinear estimates for $\mathbb{H}_{m}^{d-1}$ due to Lee, \cite{L}, and Vargas, \cite{V}. The weighted estimates are then used in a broad-narrow argument similar to \cite{DZ}, though to justify their use we need to incorporate some observations from \cite{Ba} about the estimates of Lee and Vargas. We carry out this argument and prove \eqref{bilinear} in Section \ref{five}. We remark that the cases considered in Section \ref{five} do not require any decoupling estimates beyond the trivial $L^{\infty}$ decoupling which is a consequence of Cauchy-Schwarz. 
 
 \subsection*{Acknowledgments} The authors thank Jonathan Hickman and Marina Iliopoulou for helpful discussions, from which Section 3 in particular benefited. The authors also thank the anonymous referee for helpful comments. The second author is partially supported by the Simons collaboration grant, 634269.

\section{Preliminaries}
\label{preliminaries}
We summarize and prove some important results which we will use throughout the rest of the paper. First note that $\beta(\alpha,\Gamma)$ is nondecreasing and continuous (see, e.g., Lemma 3.1 in \cite{wolff1}) in $\alpha$. In addition 
 $\beta(\alpha,\Gamma)\leq \alpha$ since, using the invariance of $\beta(\alpha,\Gamma) $ under dilations and rotations,  one can bound the energy integral. 
 
In addition it will be more convenient for us to work with the following equivalent formulation of \eqref{gammaavg}. We let $\widetilde{\beta}(\alpha,\Gamma)$ be the supremum over all $\beta$ such that $$\int_{ \Gamma }|\widehat{\mu}(R\xi)|^2 \, d\sigma(\xi) \lesssim_{\beta} c_{\alpha}(\mu) R^{-\beta}$$ for all Borel probability measures $\mu$ supported in $B^{d}(0,1)$, where  $$c_{\alpha}(\mu) := \sup_{\substack{ B(x, r) \\ r> 0, x\in \R^d } } \frac{\mu(B(x,r))}{r^{\alpha}}.$$
We note that for all $\alpha\in (0,d)$, $\widetilde{\beta}(\alpha,\Gamma) =\beta(\alpha,\Gamma)$. One direction follows from the inequality $I_{\alpha_1}(\mu)\lesssim c_{\alpha_2}(\mu) $ for $\alpha_2>\alpha_1$ and the continuity of $\beta(\alpha, \Gamma)$ in $\alpha$. For the other direction, one can use Lemma 1.5 in \cite{wolff1}.

\subsection{Elementary Positive Results} We recall the following theorem. 

\begin{theorem}[\cite{S},\cite{E}]  Let $\nu$ be a probability measure on $\R^d$ with compact support such that $$|\widehat{\nu}(\xi)| \lesssim |\xi|^{-a} \ \ \text{ and } \ \ \nu(B(x,r)) \lesssim r^{b}, \ \ x \in \R^d$$ for $a, b \in (0,d)$. Then for any $\mu\in \sm$, $$\int |\widehat{\mu}(R\xi)|^2 d\nu(\xi) \lesssim I_{\alpha}(\mu) R^{-\max( \min(a, \alpha), \alpha - d + b  ) }.$$  \end{theorem} 

\noindent In our setting we may take $\nu$ to be the surface measure of $\mathbb{H}_{m}^{d-1}$, so that $a = \frac{d-1}{2}$ and $b = d-1$ above.  This implies the lower bounds in \eqref{eq:elementary}.

\subsection{Equivalence Between Decay and Localized Restriction Estimates}\label{sec:equiv} Let $Ef$ denote the Fourier extension operator associated to $\mathbb{H}_{m}^{d-1}$. As in \cite{E2}, \cite{DGOWWZ}, \cite{DZ} we will see that it will suffice to consider certain weighted $L^{2}$ estimates for $Ef$.

Fix $\alpha \in (0, d]$ and $R > 1$. Let $\mu$ be an $\alpha$-dimensional measure supported in the unit ball as above, and let $\mu_R$ be the measure on $B^{d}(0,R)$ defined by $\mu_R (A) = R^{\alpha}\mu(R^{-1}A).$ Note that $$c_{\alpha}(\mu_R) \lesssim c_{\alpha}(\mu) $$ with the implicit constant independent of $R$. We let $s_{d}(\alpha)$ be the infinum over all $s$ such that $$\|Ef\|_{L^2 (d\mu_R; B_{R}(0,R))} \lesssim c_{\alpha}(\mu_R)^{\frac{1}{2}} R^{s} \|f\|_{L^2}, \ \ \ \ \ \ \text{supp}(f) \subset B^{d-1}(0,2)$$ for all $\alpha$-dimensional probability measures $\mu$ supported in $B^d(0,1)$. We will repeatedly use the following relationship between $\beta(\alpha, \mathbb{H}^{d-1}_m)$ and $s_{d}(\alpha)$. 

\begin{prop}\label{prop:equiv} One has $$s_{d}(\alpha) = \frac{\alpha - \beta(\alpha, \mathbb{H}^{d-1}_m) }{2}.$$ 
\end{prop}

\noindent A proof of this proposition is in the Appendix below.

\subsection{Wave Packet Decomposition}

 Fix a scale $R  > 1$ and suppose $\theta$ is a ball of radius $R^{-\frac{1}{2}}$ in frequency space. We let $G(\theta)$ denote the unit normal to $\mathbb{H}$ above the center of $\theta$. In particular, if $\xi \in \R^{d-1}$ is the center of $\theta$ then we have $$G(\theta) = \frac{1}{|(2\xi, -1)|}(2\xi_1, \dotsc, 2\xi_{d-m-1}, -2\xi_{d-m},\dotsc,-2\xi_{d-1} , - 1).$$ Note that the angle between $G(\theta_1)$ and $G(\theta_2)$ is proportional to the distance between the centers of $\theta_1$ and $\theta_2$. Below we will also let $G_{0}$ be the function $$G_0 (\xi) = (2\xi_1, \dotsc, 2\xi_{d-m-1}, -2\xi_{d-m},\dotsc,-2\xi_{d-1} , - 1). $$

 We recall the scale $R^{\frac{1}{2}}$ wave packet decomposition for $Ef$. We let $\{\theta \}$ be a collection of finitely-overlapping balls of radius $R^{-\frac{1}{2}}$ covering the support of $f$, and let $\{\nu \}$ be a collection of finitely-overlapping balls of radius $R^{\frac{1}{2}}$ covering $B_{x}^{d-1}(0,R).$ Then using a partition of unity we decompose $$f = \sum_{\theta, \nu} f_{\theta, \nu}$$ where $f_{\theta, \nu}$ is supported in a small neighborhood of the $R^{-\frac{1}{2}}$-ball $\theta$, and where $\widehat{f_{\theta, \nu}}$ rapidly decays outside the $R^{\frac{1}{2}}$-ball $\nu$. Then $$Ef = \sum_{\theta, \nu} Ef_{\theta, \nu}$$ and each wave packet $Ef_{\theta, \nu}$ is essentially supported in a $R^{\frac{1}{2} + \delta} \times \dotsm \times R^{\frac{1}{2} + \delta} \times R^{1 + \delta}$ tube $T_{\theta, \nu}$ in $\R^d$ passing through $\nu$ with long direction $G(\theta)$. Here $\delta > 0$ is a small parameter which will be harmless to our estimates, and hence we suppress its role below (we can for example take $\delta = \epsilon^{100}$, where $\epsilon$ is fixed below). For more on this wave packet decomposition see for example \cite{L}, \cite{G}.

\section{Narrow decoupling}\label{sec:narrow}  Fix a scale $K \gg 1$ and decompose the support of the input function $f$ as a union of finitely-overlapping caps $\tau$ of radius $K^{-1}$. Then use a partition of unity to decompose $f = \sum_{\tau}f_{\tau}$ and $Ef = \sum_{\tau} Ef_{\tau}$.
 
We recall the following decoupling result for surfaces with non-zero Gaussian curvature proved by Bourgain and Demeter.

\begin{prop}[\cite{BD2}] \label{prop:narrowDec0} Let $\mathcal{M}$ be a smooth, compact manifold with Gaussian curvature bounded away from 0 and let $m$ denote the minimum between the number of positive and negative principal curvatures of $\mathcal{M}$. Let $E_{\mathcal{M}}f$ denote the Fourier extension operator associated to $\mathcal{M}$. Let $Q$ be a $K^2$-cube and suppose $E_{\mathcal{M}}f = \sum_{\tau} E_{\mathcal{M}}f_{\tau}$, where the $\tau$ are $K^{-1}$-caps. Also suppose $2 \leq p \leq \frac{2(d-m + 1)}{d -m - 1}$. Then for any $\eta > 0$ one has $$\|E_{\mathcal{M}}f\|_{L^p (Q)} \lesssim_{\eta} K^{m( \frac{1}{2} - \frac{1}{p} )  + \eta} \bigg( \sum_{\tau} \|E_{\mathcal{M} }f_{\tau}\|_{L^p (w_Q) }^{2}\bigg)^{\frac{1}{2}}.$$ If $p >  \frac{2(d-m + 1)}{d -m - 1}$ then the above estimate holds with the loss $K^{m( \frac{1}{2} - \frac{1}{p} )}$ replaced by $K^{\frac{d-1}{2} - \frac{d+1}{p}}.$
\end{prop}

\noindent If $f$ is supported near a lower-dimensional space we can take advantage of the following `narrow decoupling' result for $\mathbb{H}_{m}^{d-1}$, which will be useful for proving lower bounds for $\beta(\alpha, \mathbb{H}^{d-1}_m)$. Given a $K^{-1}$-cap $\tau$, let $\omega_{\tau}$ denote the center of $\tau$. Given a subspace $V$ of $\R^{d}$ we will write $\tau \in V$ to signify that $$\text{Angle}( G(\omega_{\tau}), V) \leq K^{-1}.$$  We say that $Ef$ is concentrated along a $K^{-1}$ neighborhood of $V$ if $$\sum_{\substack{\tau \notin V } }Ef_{\tau}  \ = \ \text{RapDec}(R)\|f\|_{L^2},$$  where RapDec$(R)$ is a term such that for any $N > 1$ $$|\text{RapDec}(R)| \leq C_{N} R^{-N}.$$  

\begin{prop}[\cite{BD2}] \label{prop:narrowDec} Let $Q$ be a $K^2$-cube and suppose $Ef = \sum_{\tau} Ef_{\tau}$ is concentrated along an $O(K^{-1})$ neighborhood of a $k$-dimensional vector space $V$ in $\R^d$ with $k \geq m+1$. Also suppose $2 \leq p \leq \frac{2(k-m + 1)}{k -m - 1}$. Then for any $\eta > 0$ one has
 \begin{equation}\label{eq:narrowDecEq}
\|Ef\|_{L^p (Q)} \lesssim_{\eta} K^{m( \frac{1}{2} - \frac{1}{p} )  + \eta} \bigg( \sum_{\tau} \|Ef_{\tau}\|_{L^p (w_Q) }^{2}\bigg)^{\frac{1}{2}}
\end{equation}
\end{prop} 

\noindent Notice that $p_{m,k}:= \frac{2(k-m + 1)}{k -m - 1}$ is decreasing in $k$, hence the range of $L^{p}$ exponents where the loss scales as $K^{m(\frac{1}{2} - \frac{1}{p})}$ increases when $k$ is smaller. We will see below that this leads to improved lower bounds for $\beta(\alpha, \mathbb{H}^{d-1}_m)$ in the interesting range $\alpha < d - m$. We also remark that Proposition \ref{prop:narrowDec} is indeed an improvement over Proposition \ref{prop:narrowDec0} when applicable. For example, in the case $d = 4, m = 1, k = 3$ one can use Proposition \ref{prop:narrowDec} to decouple at $p = 6$ with a loss of $K^{\frac{1}{3}}$. However if one instead uses Proposition \ref{prop:narrowDec0} in this case with $p = 6$ the loss is $K^{\frac{2}{3}}$. 

The proof of Proposition \ref{prop:narrowDec} is implicit in the argument given in Section 3 of \cite{BD2} (in particular in Proposition 3.2 of \cite{BD2}). It is also similar to other `narrow decoupling' arguments (see for example Lemma 9.3 in \cite{G} or Section 2 in \cite{TH1}), although some new issues arise related to the presence of affine subsets of $\mathbb{H}_{m}^{d-1}$ and also the action of the Gauss map associated to $\mathbb{H}_{m}^{d-1}$. We include most of the details for the convenience of the reader, since it is worth illustrating how $\mathbb{H}_{m}^{d-1}$ differs from the case of the (elliptic) paraboloid or cone.

The main idea is the following: even though the intersection of $\mathbb{H}_{m}^{d-1}$ with a $k$-plane may have zero Gaussian curvature, there are limits to the loss of curvature in terms of the parameter $m$. Indeed, the surface $\mathbb{H}_{m}^{d-1}$ can contain affine subsets, but only of dimension less than or equal to $m$. The key quantitative tool is the following lemma.

Let $V_0$ be a $(k-1)$-dimensional subspace of $\R^{d-1}$ and let $W = V_{0} \times \R$. We set $$H_W = \mathbb{H}_{m}^{d-1} \cap W.$$ It is straightforward to check that one can parametrize $H_W$ with a quadratic form and hence the principal curvatures are constant along the surface. Let $V$ the $k$-dimensional subspace such that $G_0(\omega) \in V$ if and only if $\omega \in V_0$. 

\begin{lemma}\label{lem:eigenvalues} Let $I_{K,b} = (-1 + K^{-b}, 1 - K^{-b})$. Let $m(H_W)$ denote the minimum between the number of positive and negative principal curvatures of $H_W$ which are outside the interval $I_{K,b}$, and let $r(H_W)$ denote the number of principal curvatures which are in $I_{K,b}$. Then $$m(H_W) + r(H_W) \leq m $$ 
\end{lemma}

\begin{proof} The proof is the same as an argument given in the proof of Proposition 3.2 in \cite{BD2}. We apply a rotation to $\R^{d-1}$ to assume that we can write $$V_0 = \text{span}\{e_1, ..., e_{k-1} \}$$ where $e_{i}$ are the standard basis vectors in $\R^{d}$. This changes the defining matrix $M$ for $\mathbb{H}_{m}^{d-1}$ but of course does not change any geometric properties of the surface (and in particular the eigenvalues of the new $M$ are still $\pm 1$).
	
	Then there is a symmetric $(k-1)\times (k-1)$ matrix such that $$H_W= \{(\omega, \langle A\omega, \omega \rangle) : \omega \in V_0 \}, $$ and moreover \begin{equation} \label{eq:parametrization} \langle \omega , A \omega \rangle = \langle \omega , M \omega\rangle, \ \ \ \omega \in V_0 \end{equation} We can find an orthonormal basis of $(d-1)$ eigenvectors $\eta^i$ for $M$ in $\R^{d-1}$, and an orthonormal basis of $k-1$ eigenvectors $v^i$ for $A$ in $V_0$. Their eigenvalues are the respective principal curvatures. After relabeling we assume the eigenvectors $v^i$ are ordered based on their eigenvalues being positive and outside $I_{K,b}$, then negative and outside $I_{K,b}$, and then finally those in $I_{K,b}$ (which for all purposes we treat as if they were 0). We also assume the eigenvectors $\eta^i$ are ordered based on their eigenvalues being positive then negative. 
	
	Let $p(\mathbb{H}_{m}^{d-1})$ and $n(\mathbb{H}_{m}^{d-1})$ denote the number of positive and negative eigenvalues of $M$, respectively. Then $m = \min(p(\mathbb{H}_{m}^{d-1}) , n(\mathbb{H}_{m}^{d-1})).$ Define eigenspaces $$X_{+} = \text{span}\{\eta^i : 1 \leq i \leq p(\mathbb{H}_{m}^{d-1}) \}, \ \ \ \ X_{-} = \text{span}\{\eta^i : p(\mathbb{H}_{m}^{d-1}) + 1 \leq i \leq d-1 \} $$ and $$X^{W}_{+} = \text{span}\{v^i : 1 \leq i \leq p(H_W) \},$$ $$ X_{-}^{W} = \text{span}\{v^i : p(H_W) + 1 \leq i \leq p(H_W) + n(H_W) \},  $$ $$X_{0}^{W} = \text{span}\{v^i : p(H_W) + n(H_W) + 1 \leq i \leq k-1 \}. $$
	
	\noindent We claim that \begin{equation}\label{eq:eigenEq1} n(H_W) + r(H_W) \leq n(\mathbb{H}_{m}^{d-1}) 
	\end{equation} and \begin{equation}\label{eq:eigenEq2} p(H_W) + r(H_W) \leq p(\mathbb{H}_{m}^{d-1}) 
	\end{equation} 
	
	\noindent These follow by dimension counting. For example, suppose \eqref{eq:eigenEq1} is false. Then $(X_{-}^{W} \oplus X_{0}^{W}) \cap X_{+}$ must contain a unit vector $u$. Indeed note that dim $X_{+} = p(\mathbb{H}_{m}^{d-1})$ and the dimension of $X_{-}^{W} \oplus X_{0}^{W} $ is $n(H_W) + r(H_W)$. Then if \eqref{eq:eigenEq1} fails we have $$\text{dim}(X_{-}^{W} \oplus X_{0}^{W}) > n(\mathbb{H}_{m}^{d-1}),$$ and the claim then follows since $p(\mathbb{H}_{m}^{d-1}) + n(\mathbb{H}_{m}^{d-1}) = d-1.$ Since  $(X_{-}^{W} \oplus X_{0}^{W}) \cap X_{+}$ contains a unit vector $u$ it follows by definition that $$\langle Au , u \rangle \leq 1 - K^{-b}$$ and also $$\langle Mu , u \rangle = 1.$$ But $\langle Au , u \rangle = \langle Mu , u \rangle$ since $u \in V_0$, contradiction. The proof of \eqref{eq:eigenEq2} is similar. 
	
	Now suppose that $m(H_W) + r(H_W) > m$. Then $$p(H_W) + r(H_W) > m, \ \ \ \ \ \ n(H_W) + r(H_W) > m$$ and from \eqref{eq:eigenEq1} and \eqref{eq:eigenEq2} we then obtain $$n(\mathbb{H}_{m}^{d-1}) > m,  \ \ \ \ \ \ p(\mathbb{H}_{m}^{d-1}) >m.$$ This contradicts the definition of $m$ and so the result follows. \end{proof}

The following lemma allows us to use lower-dimensional cases of Proposition \ref{prop:narrowDec0} for intersections that have enough curvature. 

\begin{lemma}\label{lem:lowerDim} Fix $l$ with $2 \leq l \leq d$. Let $V$ be an $l$-dimensional subspace of $\mathbb{R}^d$, and suppose $F = \sum_{\tau} F_{\tau}$ is such that each $\widehat{F_{\tau}}$ has support in a $K^{-2}$ neighbourhood of $\tau$, such that the normal to $\tau$ is contained in a $K^{-1}$-neighbourhood of $V$. 
	
	Let $W  = G_0^{-1}(V) \times \mathbb{R}$. If
	\[ \mathbb{H}_m^{d-1} \cap W = \mathbb{H}_m^{d-1} \cap G^{-1}(V)\]
	is a smooth $(l-1)$-dimensional surface with nonvanishing Gaussian curvature, then 
	\[ \pi_V(\supp \widehat{F} ) \subseteq \mathcal{N}_{CK^{-2}}( \pi_V(\mathbb{H}_m^{d-1}  \cap G^{-1}(V))), \]
	and $\pi_V(\mathbb{H}_m^{d-1}  \cap G^{-1}(V))$ is a smooth $(l-1)$-dimensional surface in $V$ with nonvanishing Gaussian curvature. Moreover, the sets 
	\[\mathcal{N}_{CK^{-2}}(\pi_V(\supp \widehat{F_{\tau}} )) \]
	are essentially disjoint. All implicit constants depend only on the lower bound for the magnitude of the Gaussian curvature of $\mathbb{H}_m^{d-1}  \cap G^{-1}(V)$.
\end{lemma}
\begin{proof} Let $W_0$ be the unique $l$-dimensional subspace of $\mathbb{R}^d$ parallel to $W$. It will first be shown that $\pi_{W_0}: V \to W_0$ is bi-Lipschitz. 
	The composition $\pi_{W_0} \circ G$ is nonvanishing since $e_d \in W$, and so
	\[ G_{W_0} = \frac{ \pi_{W_0} \circ G}{|\pi_{W_0} \circ G|}, \]
	where $G_{W_0}: \mathbb{H}_m^{d-1} \cap G^{-1}(V) \to S^{d-1} \cap W_0$ is the Gauss map on $\mathbb{H}_m^{d-1} \cap G^{-1}(V)$. This will be used to show that
	\begin{equation} \label{biLipschitz} |\pi_{W_0}(v) | \sim |v|, \end{equation}
	for every $v \in V$. Suppose for a contradiction that \eqref{biLipschitz} fails. Then by compactness there exists $v \in V \cap S^{d-1}$ such that 
	\[ |\pi_{W_0}(v)| = 0. \]
	Then $\pi_{W_0}$ maps $V$ into a subspace $E \subseteq W_0$ of dimension $< l$, and so the image of $G_{W_0}: \mathbb{H}_m^{d-1} \cap G^{-1}(V) \to S^{d-1} \cap W_0$ is contained in $ S^{d-1} \cap E$, which has dimension $< l-1$. But $G_{W_0}$ is locally bi-Lipschitz since $\mathbb{H}_m^{d-1} \cap G^{-1}(V)$ has nonvanishing Gaussian curvature, so this is a contradiction. 
	
	This shows that $\pi_{W_0}: V \to W_0$ is bi-Lipschitz, and this implies that $\pi_V : W_0 \to V$ is bi-Lipschitz, since
	\[ |w|^2 = \left\langle  \pi_V(w), \left( \pi_{W_0}|_V \right)^{-1}(w) \right\rangle \lesssim |\pi_V(w)|  |w|, \]
	for any $w \in W_0$. Therefore $\pi_V : W \to V$ is bi-Lipschitz, and by compactness the bi-Lipschitz constant depends only on the lower bound of the Gaussian curvature of $\mathbb{H}_m^{d-1} \cap G^{-1}(V)$. This implies that $\pi_V(\mathbb{H}_m^{d-1}  \cap G^{-1}(V))$ is a smooth $(l-1)$-dimensional surface in $V$ with nonvanishing Gaussian curvature, and the sets 
	\[\mathcal{N}_{CK^{-2}}(\pi_V(\supp \widehat{F_{\tau}} )) \]
	are essentially disjoint. 
	
	For each cap $\tau$, there exists a point $x \in C_2\tau$ with $G(x) \in V$, since $G$ is locally bi-Lipschitz, where $C_2$ is a large constant. The tangent plane at $x$ satisfies 
	\begin{align*} T_{\pi_V(x)}\pi_V(\mathbb{H}^{d-1}_m \cap W) &= \pi_V( T_x(\mathbb{H}^{d-1}_m \cap W)) \\
	&= \pi_V( T_x(\mathbb{H}^{d-1}_m)). \end{align*}
	The second line follows from the fact that $\pi_V( T_x(\mathbb{H}^{d-1}_m))$ is only $(l-1)$-dimensional (since $G(x) \in V$), and contains $\pi_V(T_x(\mathbb{H}^{d-1}_m \cap W))$ which is also $(l-1)$-dimensional. Let $ T_x'(\mathbb{H}^{d-1}_m) = x+  T_x(\mathbb{H}^{d-1}_m)$. Then
	\begin{align*} \pi_V\left(\supp \widehat{F_{\tau}}\right) &\subseteq \pi_V( \mathcal{N}_{K^{-2}}(\tau)) \\
	& \subseteq \mathcal{N}_{K^{-2}}(\pi_V(\tau)) \\
	& \subseteq \mathcal{N}_{C_1K^{-2}}(\pi_V( T_x' \mathbb{H}^{d-1}_m \cap B(x,C_1K^{-1}) ) ) \\
	&\subseteq \mathcal{N}_{C_1K^{-2}}(\pi_V( T_x' \mathbb{H}^{d-1}_m) \cap B_V(\pi_V(x),C_1K^{-1}) ) \\
	&= \mathcal{N}_{C_1K^{-2}}(T_{\pi_V(x)}' \pi_V(\mathbb{H}^{d-1}_m \cap W) \cap B_V(\pi_V(x),C_1K^{-1}) ) \\
	&\subseteq \mathcal{N}_{CK^{-2}}(\pi_V(\mathbb{H}^{d-1}_m \cap W)) \\
	&= \mathcal{N}_{CK^{-2}}(\pi_V(\mathbb{H}^{d-1}_m \cap G^{-1}(V))). \end{align*} 
	This finishes the proof.  \end{proof} 

\begin{remark} It is possible that the intersection $\mathbb{H}_{m}^{d-1} \cap W$ has Gaussian curvature near 0, in which case the conclusion of Lemma \ref{lem:lowerDim} can fail. For example, suppose for simplicity we are in the case $d = 3$ and after applying a rotation assume the phase is of the form $\xi_3 = \xi_1 \xi_2$. Suppose the normals are contained in a $K^{-1}$-neighborhood of the vector space $$V = \{\xi \in \R^3 : \xi_1 = 0 \}.$$ Then if $$V_0 = \{\xi \in \R^2 : \xi_2 = 0  \}$$ it follows that the input function is supported in an $O(K^{-1})$-neighborhood of $V_{0}$. In this case the projection to $V$ of the support of $\widehat{Ef}$ may not be contained in an $O(K^{-2})$-neighborhood of $\pi_V(H_W)$ (where as before $W = V_0 \times \R$). For example, if $f \sim 1$ near the region where $|\xi_1| \sim 1$ then the projection of the support of $\widehat{Ef}$ is spread out in the interval where $|\xi_3| \leq K^{-1}$. Note however that in this case the intersection $H_W$ has zero Gaussian curvature and the projection $\pi_V$ is not bi-Lipschitz.  

If the intersection $H_W$ has Gaussian curvature bounded above by some $K^{-\sigma}$ then similar arguments show that the conclusion of Lemma \ref{lem:lowerDim} can fail. However Lemma \ref{lem:eigenvalues} will allow us to choose our slices so that the intersection has principal curvatures near 1 in absolute value, avoiding this issue. 
\end{remark}

\begin{remark}\label{rmk:lip} The dependence of the Lipschitz constant on the curvature in Lemma \ref{lem:lowerDim} can be made more quantitative. Assume $W$ is a subspace (which we can in application) and let $A$ be the symmetric $(k-1)\times (k-1)$ matrix such that (in appropriate coordinates) $$ \langle A \omega, \omega \rangle = \langle M \omega , \omega \rangle, \ \ \ \omega \in W \cap \R^{d-1}.$$ Now let $\{w_1,...,w_{k-1} \}$ be an orthonormal basis of $W\cap \R^{d-1}$ consisting of eigenvectors for $A$. Then $\{Mw_1, ..., Mw_{k-1}, e_d \}$ is an orthonormal basis for $V$ since $$Mw_{i} - e_d = G_0 (w_i) \in V.$$ 
	
	Since $\langle Aw, w \rangle = \langle Mw, w\rangle$ when $w \in W \cap \R^{d-1}$ we must have $$\langle M(w_i - w_j), w_i - w_j \rangle = \langle A(w_i - w_j), w_i - w_j \rangle $$ and hence $\langle Mw_j, w_i \rangle = \langle A w_j, w_i \rangle$. Now if $v = \sum_{i} v_i Mw_i + v_d e_d \in V$ then \begin{align*} \pi_{W}(v) &= \sum_{j}  \big(\sum_{i} v_i \langle Mw_i, w_j \rangle \big) w_j + v_d e_d \\ &= \sum_{j}  \big(\sum_{i} v_i \langle Aw_i, w_j \rangle \big) w_j + v_d e_d  \\ &= \sum_{j}  (\lambda_j v_j) w_j + v_d e_d   \end{align*} where the $\lambda_j$ are the eigenvalues. Since the eigenvalues are bounded in absolute value by 1 it follows that $$|\pi_{W}(v)|^2 = \sum_{j}\lambda_{j}^2 v_{j}^2 + v_{d}^2 \geq (\inf_{j} |\lambda_j| )^{2}|v|^2.$$ Therefore $$|\pi_{W}(v)| \geq (\inf_{j} |\lambda_j| ) |v| $$ where the $\lambda_j$ are the principal curvatures of $H_W$ in $W$. As a consequence the bi-Lipschitz constant of $\pi_V : W \rightarrow V$ is also bounded away from 0, and in particular $$|\pi_V (w)| \geq (\inf_{j} |\lambda_j|)|w|, \ \ \ w \in W.$$ 
\end{remark}

Finally we recall the following `trivial' decoupling result which allows us to eliminate directions with small curvature. 

\begin{prop}[Flat Decoupling] \label{prop:flatDecoupling} Suppose $\mathcal{T}$ is a collection of finitely-overlapping and parallel rectangles $S$ in $B^{d}(0,2)$. Let $F = \sum_{S} F_{S}$ with $\widehat{F}_{S}$ supported in $S$. Then one has $$\| F \|_{L^{p}(\R^d)} \leq C(\# \mathcal{T})^{\frac{1}{2} - \frac{1}{p}}  \big( \sum_{S \in \mathcal{T} } \|F_{S}\|^{2}_{L^{p}(\R^d) } \big)^{\frac{1}{2}}.$$
\end{prop}

\begin{proof} The case $p = \infty$ is just the Cauchy-Schwarz inequality, and when $p =2$ the proposition follows from Plancherel's theorem. The remaining cases follow by vector-valued interpolation (as in Lemma 4.5 in \cite{De}, for example).
	
Indeed, notice that if $\phi_S$ are Schwartz functions with $\widehat{\phi_S} = 1$ on $S$ then $F = \sum_{S} F_{S} \ast \phi_{S}$. Let $\mathcal{S}$ denote the collection of families of functions of the form $f = \{f_S\}_{S \in \mathcal{T}}$. If we define an operator on $\mathcal{S}$ by $Tf = \sum_{S}f_{S} \ast \phi_{S}$ then the argument summarized above shows that $T$ is bounded from $\ell_{\mathcal{S}}^2 L^{\infty}$ to $L^{\infty}$ with operator norm $(\# \mathcal{T})^{\frac{1}{2}},$ and also bounded from $\ell_{\mathcal{S}}^2 L^{2}$ to $L^{2}$ with operator norm $O(1)$; moreover the operator norms only depend on the amount of overlap (which is $O(1)$) and are otherwise independent of the collection $\mathcal{T}$. Since $F = T\{ F_S\}_{S \in \mathcal{T}}$ the desired result then follows after interpolating the above estimates for $T$ in the vector-valued setting. 
\end{proof}

\noindent Suppose the rectangles $S$ are $O(K^{-1})$-neighborhoods of lower-dimensional rectangles in $\R^{d-1} \times \R$ and $F = Ef \cdot w_{Q}$, where $Q$ is a $K^{2}$-cube and $w_Q$ is a smooth weight adapted to $Q$. Then the above proposition can be localized to $Q$ by choosing $F_S$ to be concentrated in $Q$, and we will use this localized version below. This can be achieved for example by taking $F_S$ to consist of scale-$K$ wave packets which are concentrated near $S \cap \mathbb{H}^{d-1}_{m}$ in Fourier space and concentrated in $Q$ spatially.

\begin{proof}[Proof of Proposition \ref{prop:narrowDec}] Fix a $k$-dimensional plane $V$ as above and let $W$ be the $k$-plane $W = V_0 \times \R$, where $V_0 = G_{0}^{-1}(V)$. Using affine invariance assume that $V_0$ is a subspace. To simplify notation let $F$ denote $Ef\cdot w_{B_{K^2 }}$. We let $H_W = \mathbb{H}_{m}^{d-1} \cap W$ and define $m(H_W)$ and $r(H_W)$ as in the statement of Lemma \ref{lem:eigenvalues}. 
	
	We identify $W$ with $\R^{k}$. As above we can parameterize $H_W$ as the graph of a possibly degenerate quadratic form whose defining matrix $A$ is symmetric. We may find an orthonormal basis of $\R^k$ consisting of eigenvectors for $A$, and the respective eigenvalues $\lambda_{i}$ are the principal curvatures of $H_W$. We let $r(H_W)$ denote the number of these eigenvalues in the interval $I_{K,2}$. We perform a flat decoupling (Proposition \ref{prop:flatDecoupling}) in the directions of the eigenvectors with eigenvalues inside $I_{K,2}$, contributing a loss of $$K^{r(H_W)(\frac{1}{2} - \frac{1}{p} ) }.$$ We claim that by Lemma \ref{lem:eigenvalues} the resulting slices are $O(K^{-1})$-neighborhoods of $(k-r(H_W))$-planes $W_r$ such that $W_r \cap \mathbb{H}_{m}^{d-1} := H_{W_r}$ is a surface of dimension $k - r(H_W) - 1$ with principal curvatures bounded below in absolute value by $1- K^{-2}$.
	
	More precisely, after applying a rotation we can assume that the standard basis vectors $e_1, ... e_{k-1}$ are eigenvectors for $A$ and moreover by Lemma \ref{lem:eigenvalues} that $$\{\lambda_{1 }, ..., \lambda_{p(H_W) + n(H_W)} \} \subset I_{K,2}^{c}.$$ In these coordinates we let $W_{r,0}$ be the subspace $$W_{r,0} = \{\omega \in V_0 :   \omega_{p(H_W) + n(H_W) + 1} = ... = \omega_{k-1} = 0\}$$ and let $\{W_r\}$ be the family of $K^{-1}$-separated planes in $\R^{d-1}$ obtained by translating $W_{r, 0} \times \R$ in axis-parallel directions. Then after applying a flat decoupling in each of the $r(H_W)$ directions $e_{p(H_W) + n(H_W) + 1}, ..., e_{k-1}$ we can assume that $\widehat{F}$ is supported in a $K^{-1}$-neighborhood of one of the $W_r$. Since $\widehat{F}$ is supported in $B^{d-1}(0,2)$ this contributes a loss of $K^{r(H_W)(\frac{1}{2} - \frac{1}{p} ) }$ to our main estimate, as claimed above. Moreover, $H_{W_r}$ is parametrized by a quadratic form whose defining matrix has each of its $k - r(H_W) -1$ eigenvalues outside of $I_{K,2}$. Hence $H_{W_r}$ is a smooth surface of dimension $k - r(H_W) - 1$ and the claimed lower bounds on the principal curvatures follow. 
	
	Now let $F_{W_{r}}$ denote the Fourier restriction of $F$ to a $K^{-1}$-neighborhood of $W_r$. Let $V_r$ be the subspace of dimension $k - r(H_V)$ spanned by vectors in $G(W_r)$. By Lemma \ref{lem:lowerDim} we therefore know that if we restrict $F_{W_r}$ to a plane $V_{r}'$ parallel to $V_r$ then the Fourier transform of $(F_{W_r})|_{V_{r}'}$ is supported in an $O(K^{-2})$ neighborhood of the projection of $H_{W_r}$ to $V_r$. Moreover one checks using an argument similar to Remark \ref{rmk:lip} that the signs of the curvatures of the surface are preserved by the projection. Then by Proposition \ref{prop:narrowDec0} we can decouple the support of $\widehat{F}|_{N_{K^{-1}}(W_r) }$ into $K^{-1}$ caps with a loss of $$C_{\eta}K^{m(H_W)(\frac{1}{2} - \frac{1}{p}) + \eta } $$ as long as $$2 \leq p \leq \frac{ 2((k - r(H_W)) - m(H_W) + 1) }{(k - r(H_W)) - m(H_W) - 1}.$$ We can stitch together these steps in the usual way using Fubini's theorem and Minkowski's inequality to obtain 
	
	\begin{equation}\label{eq:narrowDecEq2}
	\|Ef\|_{L^p (Q)} \lesssim_{\eta} K^{[m(H_W) + r(H_W)]( \frac{1}{2} - \frac{1}{p} )  + \eta} \bigg( \sum_{\tau} \|Ef_{\tau}\|_{L^p (w_Q) }^{2}\bigg)^{\frac{1}{2}}
	\end{equation} as long as $p \leq \frac{ 2((k - r(H_W)) - m(H_W) + 1) }{(k - r(H_W)) - m(H_W) - 1}$ (for similar `slicing' arguments see for example Lemma 9.3 in \cite{G} or Section 2 in \cite{TH1}).
	
	The argument is complete if $m(H_W) + r(H_W) = m$, so suppose $$m(H_W) + r(H_W) < m$$ (this is the only remaining case by Lemma \ref{lem:eigenvalues}). Note that in this case $$\frac{2(k-m +1)}{k-m-1} >  \frac{ 2((k - r(H_W)) - m(H_W) + 1) }{(k - r(H_W)) - m(H_W) - 1}.$$By interpolating between \eqref{eq:narrowDecEq2} and the trivial Cauchy-Schwarz estimate $$ \|Ef\|_{L^{\infty} (Q)} \lesssim K^{ \frac{k-1}{2} } \bigg( \sum_{\tau \in V} \|Ef_{\tau}\|_{L^{\infty} (w_Q) }^{2}\bigg)^{\frac{1}{2}}$$ we obtain $$\|Ef\|_{L^p (Q)} \lesssim_{\eta} K^{ \frac{k-1}{2} - \frac{(k+1)}{p} + \eta} \bigg( \sum_{\tau} \|Ef_{\tau}\|_{L^p (w_Q) }^{2}\bigg)^{\frac{1}{2}}, \ \ \ \ p = \frac{2(k - m+ 1)}{k - m - 1}. $$ A bit of algebra then shows that $$\frac{k-1}{2} - \frac{(k+1)}{p} = m (\frac{1}{2} - \frac{1}{p} ) , \ \ \ \ \ p = \frac{2(k - m + 1)}{k - m - 1}.$$ This completes the proof.

\end{proof}

\section{Upper bounds for \texorpdfstring{$\beta(\alpha, \mathbb{H}^{d-1}_m)$}{beta}}
\label{three}
By Proposition \ref{prop:equiv}, counterexamples to localized weighted restriction estimates imply upper bounds on $\beta(\alpha, \mathbb{H}^{d-1}_m)$. It is convenient to work with a discretized version of the weighted restriction estimates. Let $X$ be a union of unit lattice cubes in $B^{d} (0, R)$. We abuse notation and write $Q \in X$ if $Q$ is a lattice unit cube with $Q \subset X$. We say $X$ is $\alpha$-dimensional if $$\gamma^{d}_{\alpha}(X) := \sup_{\substack{ B^{d}(z, r) \\ z \in \R^n, \  r \geq 1 }} \frac{\# \{ Q \in X : Q \subset B^{d}(z,r) \} }{r^{\alpha}} \in [c, C]$$ with $c,C$ independent of $R$. We let $\overline{s_{d}}(\alpha)$ be the infimum over all $s\geq 0$ such that $$\|Ef\|_{L^2(X)} \lesssim_s R^{s}\|f\|_{L^2}, \ \ \ \ \ \ \text{supp}(f) \subset B^{d-1}(0,2)$$ for all $\alpha$-dimensional $X$ contained in $B^{d}(0,R).$ 

\begin{lemma}\label{lem:latticeCube} One has $$s_{d}(\alpha) = \overline{s_d}(\alpha).$$
\end{lemma} 

\noindent A proof of this lemma can be found in the Appendix.

\subsection{A counterexample based on signs of principal curvatures} Recall that $p_M$ is the number of positive entries in $M$ and $n_M$ is the number of negative entries in $M$. Also recall $$m = \min(p_M, n_M).$$ We prove the following \begin{prop}\label{prop:upperBound} Suppose $\alpha \in [d-m-1, d -m]$. Then $$s_{d}(\alpha) \geq \frac{\alpha-m}{2(d-2m)}$$ and therefore $$\beta(\alpha, \mathbb{H}^{d-1}_m) \leq \alpha - \frac{\alpha - m}{d-2m}. $$
\end{prop}
Our proof uses a tensor product argument that takes advantage of some recent lower bounds in the elliptic case due to Du \cite{D}. We note however that if $\alpha \geq d-m$ the argument is simpler and one can take $h$ below to be a Knapp example.

 \begin{proof} After a change of variables we may assume that $$Ef(x,t) = \int_{\R^{d-1}}f(\xi) e^{2\pi i (x\cdot \xi + t\phi(\xi) ) } d\xi $$ where $$\phi(\xi) = \pm\xi_1 \xi_2 \pm \dotsb \pm \xi_{2m - 1}\xi_{2m} + \xi^{2}_{2m + 1} + \dotsb + \xi_{d-1}^{2}.$$ We let $$\Lambda = ([0, R^{-1}]\times [0,1])^{m}$$ and set $$g(\xi_1, \xi_2, \dotsc, \xi_{2m}) = \mathbbm{1}_{\Lambda}(\xi_1, \xi_2, \dotsc, \xi_{2m}).$$ We choose our input function so that $$f(\xi) = g(\xi_1,\dotsc,\xi_{2m})h(\xi_{2m +1}, \dotsc, \xi_{d - 1}).$$ If we let $\widetilde{x} = (x_1, \dotsc, x_{2m})$ and $x' = (x_{2m+1},\dotsc,x_{d-1})$ then we have $$Ef(x,t) = \widetilde{E}g(\tilde{x}, t)e^{it\Delta_{\R^{d-2m-1}}}\check{h}(x') ,$$ where $$\widetilde{E}g(\widetilde{x}, t) = \int_{\R^{2m}} g(\eta) e^{2\pi i (\widetilde{x} \cdot \eta + t( \eta_1 \eta_2 \pm \dotsb \pm \eta_{2m-1}\eta_{2m} ) )} d\eta.$$

Let $\widetilde{S}_R = ([0,cR]\times [0,c])^{m}$ for some small but uniform $c > 0$. For our choice of $g$ one easily checks that if $c$ is small enough (independent of $R$) then \begin{equation}\label{eq:lowerBound1} |\widetilde{E}g(\widetilde{x}, t)| \gtrsim R^{-m}, \ \ \ \ \ \ |t|\leq cR, \  \widetilde{x} \in \widetilde{S}_R.\end{equation} We choose our measure to be $\mu_{R} = \nu_{h} \times \nu_e$, where $\nu_{h}$ is a $\lambda$-dimensional measure on $\R^{2m}$ and $\nu_{e}$ is a $\sigma$-dimensional measure on $\R^{d-2m}$, with $\alpha = \lambda + \sigma$. Then $\mu_{R}$ is $\alpha$-dimensional. In fact it suffices to set $\lambda = m$ and let $\nu_{h}$ be a dilate of $m$-dimensional Lebesgue measure. In particular let $\widetilde{S} \subset \R^{2m}$ be the subspace spanned by $x_1, x_3,\dotsc,x_{2m-1}$. We define $\nu_{h}$ such that $\nu_{h}(Q) = 1$ for each lattice unit cube $Q$ in $\R^{2m}$ with $Q \cap \widetilde{S} \neq \emptyset,$ and $\nu_{h}(Q) = 0$ for all other lattice unit cubes. Note that we then have $\nu_{h}(\widetilde{S}_R) \sim R^{m}.$ 

By \eqref{eq:lowerBound1} we have $$\frac{ \|Ef\|_{L^{2}(\mu_R)} }{\|f\|_{L^2}} \gtrsim R^{-\frac{m}{2}} \nu_{h}(\widetilde{S}_R)^{\frac{1}{2}}\frac{ \|e^{it\Delta}\check{h}\|_{L^{2}(\nu_{e})} }{\|h\|_{L^2}} \gtrsim \frac{ \|e^{it\Delta}\check{h}\|_{L^{2}(\nu_{e})} }{\|h\|_{L^2}}.$$ We now appeal to the following lower bounds for the parabolic case due to Du:

\begin{theorem}[\cite{D}] There exists a function $h$ supported in $B^{n-1}(0,2)$ and a $\sigma$-dimensional measure $\nu$ on $\R^{n}$ supported in $B^n(0,R)$ such that $$\frac{\|e^{it\Delta_{\R^n}}\check{h}\|_{L^{2}(\nu)} }{\|h\|_{L^2}} \gtrsim R^{\frac{1 - 2\kappa(\sigma, n) }{2}},$$ where $$\kappa(\sigma, n) = \frac{n-\sigma}{2n}, \ \ \ \ \ \sigma \in [n-1, n].$$
\end{theorem} \noindent This theorem follows from Theorem 1.2(a) in \cite{D} along with Proposition \ref{prop:equiv}. 

Now recall that $\alpha \in [d-m-1, d-m]$, and therefore $\sigma \in [d-2m - 1, d-2m]$. We let $h$ and $\nu_{e}$ be the function and measure determined by the above theorem with $n = d - 2m$, thus obtaining  $$\frac{ \|Ef\|_{L^{2}(\mu_R)} }{\|f\|_{L^2}} \gtrsim R^{\frac{1 - 2\kappa(\sigma, d-2m) }{2}} = R^{\frac{\sigma}{2n}} = R^{\frac{\alpha -m}{2(d-2m)}}.$$

 \end{proof} 

As a corollary we see that if $\alpha \geq d-m$ then $s_{d}(\alpha) = \frac{1}{2} $ and therefore for such $\alpha$ we have $\beta(\alpha, \mathbb{H}^{d-1}_m) = \alpha - 1$. This follows from the monotonicity in $\alpha$ of the quantity $s_{d}(\alpha)$, along with the easy observation that $s_{d}(d) = \frac{1}{2}$ (which follows from Plancherel's Theorem). The monotonicity of $s_{d}(\alpha)$ can be seen directly from Lemma \ref{lem:latticeCube}, since monotonicity of $\bar{s}_{d}(\alpha)$ is obvious. 

\begin{remark} By considering the other cases covered by Theorem 1.1 in \cite{D} one can obtain further lower bounds for $s_{d}(\alpha)$ in the range $\alpha < d- m -1,$ and hence further upper bounds for $\beta(\alpha, \mathbb{H}^{d-1}_m)$ in this range. The argument is the same as we saw above, with the only change being the admissible value of $\kappa(\sigma, n)$ determined by Theorem 1.1 in \cite{D}. \end{remark}

\section{Lower bounds for \texorpdfstring{$\beta(\alpha, \mathbb{H}^{d-1}_m)$}{beta}: the Du-Zhang Method}
\label{four} We can obtain lower bounds for $\beta(\alpha, \mathbb{H}^{d-1}_m)$ by adapting the broad-narrow analysis of Du and Zhang \cite{DZ}. The set-up and structure of the argument are essentially the same as in \cite{DZ}, although there are a few important differences. We cannot use the stronger decoupling result for paraboloids, so we instead must adapt the argument to the weaker decoupling results that exist for surfaces with principal curvatures of mixed signs as summarized in Section \ref{sec:narrow}. On the other hand, in the `broad' case where one uses multilinear restriction estimates the curvature is less relevant and by following the argument in \cite{DZ} we actually get estimates for the broad term which are better than optimal; hence we can refine the argument by weakening the broadness assumption and consequently leaving more room to gain from narrow decoupling.

We will prove the following analogue of Proposition 3.1 in \cite{DZ}. Recall that $p_{m,k} = \frac{2(k -m + 1)}{k - m - 1}$.

\begin{sloppypar}\begin{prop}\label{prop:mainInductive} Fix any $\epsilon > 0$ and pick $\delta > 0$ with $\delta \ll \epsilon$ (say $\delta = \epsilon^4$). Let $K = R^{\delta}$ and let $\mathcal{Q} = \{Q_j\}_{j=1}^{M}$ be a collection of $K^2$-cubes in $B^d (0,R)$. Let $Y = \bigcup_{j=1}^{M}Q_j$ and $$\gamma = \sup_{\substack{ B^{d}(z,r) \\ r\geq K^2  } } \frac{\#\{Q \in Y : Q\subset B^{d}(z,r)  \}  }{r^{\alpha}}. $$ Fix $k \geq m+1$ and suppose that $\|Ef\|_{L^{p_{m,k}}(Q_j) }$ is dyadically constant as $Q_{j} \in \mathcal{Q}$ vary. 
	
 If $$ s(\alpha) \geq \max\bigg(\frac{(d - m) - \alpha}{2p_{m,k}} + \frac{\alpha - (d-m) +2 }{4}, \ \frac{\alpha}{2(k+1)} \bigg)$$ then there is $C_{\epsilon}$ such that \begin{equation}\label{eq:mainInductiveEst} \|Ef\|_{L^{p_{m,k}} (Y)} \leq C_{\epsilon} R^{s(\alpha) + \epsilon}M^{-(\frac{1}{2}  - \frac{1}{p_{m,k}})}\gamma^{\frac{1}{2}  - \frac{1}{p_{m,k}}}\|f\|_{L^{2}}\end{equation} whenever $f$ is supported in the unit ball. 
\end{prop} \noindent  After pigeonholing and using H\"{o}lder's inequality Proposition \ref{prop:mainInductive} implies that if $X$ is any $\alpha$-dimensional collection of unit cubes in $B^{d}(0,R)$ then \begin{equation}\label{eq:DZest} \|Ef\|_{L^{2}(X)} \lesssim_{\epsilon} R^{s(\alpha) + \epsilon }\|f\|_{L^2}.
\end{equation} 

\noindent Via Proposition \ref{prop:equiv} this will imply the remaining claimed lower bounds for $\beta(\alpha, \mathbb{H}^{d-1}_m)$ in all cases except those in \eqref{bilinear}. This implication is discussed further below, after the proof of the proposition. \end{sloppypar}

\begin{remark} Du and Zhang introduce an extra parameter $\lambda$ to account for the number of unit cubes in $X$ that intersect a given lattice $R^{\frac{1}{2}}$ cube. This allows them to take advantage of certain refined Strichartz estimates from \cite{DGLZ}, although as remarked in \cite{DZ} the parameter $\lambda$ is not needed for the proof of their main estimate Corollary 1.7. Although one can prove certain weaker refined Strichartz estimates for $\mathbb{H}_{m}^{d-1}$ by following arguments from \cite{DGLZ} these do not improve any of our estimates for $s_{d}(\alpha)$ below. For this reason we have chosen to prove the weaker version of Proposition \ref{prop:mainInductive} without the parameter $\lambda$.  
\end{remark}

We fix $\epsilon > 0$ for the rest of the argument. Let $\delta > 0$ be another small parameter with $\delta \ll \epsilon$ and set $$K = R^{\delta}.$$ Let $\mathcal{T}$ be a collection of $K^{-1}$-cubes tiling the support of $f$ and use a partition of unity to decompose $f = \sum_{\tau}f_{\tau}$ with $f_{\tau}$ supported in (a small dilate of) $\tau$. Also let $\mathcal{S}$ be a collection of $K^2$-cubes tiling $B^{d}(0,R)$. Given a $K^2$-cube $Q$ in $\mathcal{S}$ we define its significant set $$\mathcal{S}_{p}(Q) = \left\{ \tau : \|Ef_{\tau} \|_{L^p (Q)} \geq \frac{1}{100 (\# \mathcal{T}) } \|Ef\|_{L^p (Q)} \right\}.$$ Note that  $$ \sum_{\tau \notin S_{p}(Q) } \|Ef_{\tau}\|_{L^{p} (Q)} \leq \frac{1}{100}\|Ef\|_{L^{p} (Q)}, $$ so we may absorb terms involving $\tau$ which are not significant to the left-hand side of our estimates of $\|Ef\|_{L^{p}(Q)}$ below. In particular we can restrict attention to caps in the significant set in the analysis below. 

We say that a  $K^{2}$-cube $Q$ is $k$-\textit{narrow }and write $Q \in \mathcal{N}(k)$ if there is a $k$-dimensional subspace $V$ such that $$\text{Angle}(G(\tau), V) \leq \frac{1}{100d K}$$ for all $\tau \in \mathcal{S}_{p}(Q)$, where $G(\tau)$ is the unit normal to the surface $\mathbb{H}_{m}^{d-1}$ above the center of $\tau$. If a cube $Q$ is not $k$-narrow then we say it is $(k+1)$-\textit{broad} and write $Q \in \mathcal{B}(k+1)$. 

We proceed by induction on $R$, the case $R \sim 1$ being easy. Since we are assuming that $\|Ef\|_{L^p (Q)}$ is dyadically constant as $Q \in \mathcal{Q}$ varies it suffices to consider separately the cases where all cubes are $k$-narrow and where all cubes are $(k+1)$-broad.

\subsection{The narrow case} We begin by assuming all of the cubes are $k$-narrow. The following lemma is the main estimate in this case. To simplify notation we will set $p = p_{m,k}$ for the rest of the section. 

 \begin{lemma}\label{lem:narrowEst} Suppose that all of the cubes in $Y$ are $k$-narrow and that Proposition \ref{prop:mainInductive} is true at scale $R/K^2$. Let $$ q(\alpha,p) = (\alpha + m + 1)\left(\frac{1}{2} - \frac{1}{p}\right) + \frac{d+1}{p} - \frac{d-1}{2} -2s(\alpha).$$ Then for $p = p_{m,k}$ one has 
\[  \|Ef\|_{L^p (Y)} \leq C_{\epsilon} (\log R)^c K^{\epsilon^4 - \epsilon} K^{q(\alpha,p)}R^{s(\alpha) + \epsilon} (\gamma M^{-1})^{\frac{1}{2} - \frac{1}{p} } \|f\|_{L^2}. \] \end{lemma} 

\noindent Lemma \ref{lem:narrowEst} implies the narrow case of Proposition \ref{prop:mainInductive} as long as
\begin{equation}\label{eq:narrowInduction} s(\alpha) = \frac{(d - m) - \alpha}{2p} + \frac{\alpha - (d-m) +2 }{4}, \end{equation} which is the best choice of $s(\alpha)$ for which the induction closes. 

If we set $p = p_{m, d-1}$, corresponding to the usual narrow case $Q \in \mathcal{N}(d-1)$ , we get $$s(\alpha) = \frac{\alpha}{2(d-m)}.$$ This is in general an improvement from what we can obtain from more elementary arguments. However, if we assume $Q \in \mathcal{N}(k)$ for $k < d-1$ then we can use Proposition \ref{prop:narrowDec} to decouple with larger $p = p_{m,k}$, and thus from \eqref{eq:narrowInduction} we will have better estimates in some cases. We defer this analysis until after we sketch the proof of Lemma \ref{lem:narrowEst} and consider the broad case of Proposition \ref{prop:mainInductive}. 

\subsubsection{Proof sketch of Lemma \ref{lem:narrowEst}} The argument is essentially the same as the proof of the narrow case in \cite{DZ}, with the exception that we use the $k$-narrow decoupling in Proposition \ref{prop:narrowDec} in place of $(d-1)$-narrow decoupling for the paraboloid.  

We break $B^{d-1}(0,R)$ into $R/K$ cubes $D$ and decompose $$Ef = \sum_{(\tau, D)} Ef_{\Box_{\tau, D}},$$ where $\widehat{Ef}_{\Box_{\tau,D}} $ is supported in a small dilate of $\tau$ and $Ef_{\Box_{\tau, D}}$ decays rapidly outside a (small dilate of) an $R/K \times \dotsm \times R/K \times R$ rectangle with long direction $G(\tau)$. Since we are in the narrow case it follows that the $\tau$ are supported in an $O(K^{-1})$ neighborhood of a $(k-1)$-plane $V_0$, and hence we can decouple using Proposition \ref{prop:narrowDec} and then hope to use the induction hypothesis and parabolic rescaling. As in \cite{DZ} we first need to pigeonhole to fix certain parameters before carrying out this argument. 

 Let $R_1 = R/K^2$ and $K_1 = R_{1}^{\delta}$. By induction we can assume Proposition \ref{prop:mainInductive} holds at scale $R_1$. To take advantage of this we cover each $\Box = \Box_{\tau, D}$ by $KK_{1}^2 \times \dotsm \times KK_{1}^2 \times K^2 K_{1}^2$ tubes $S$ with long direction parallel to $G(\tau)$. We may throw away the $S$ which do not intersect cubes from $Y$. After a parabolic rescaling note that the $S$ become $K_{1}^{2}$-cubes. Now fix one box $\Box$. After dyadic pigeonholing we may assume that $\|Ef_{\Box}\|_{L^{p}(S)}$ is dyadically constant as $S$ varies. We may also assume that each $S$ contains $\sim \eta$ narrow $K^{2}$ cubes from $Y$ and that the number of $S$ in $\Box$ is $\sim M_1$.

 After further pigeonholing we may assume that the fixed parameters from the last paragraph are the same as $\Box$ varies. We can also assume that $\|f_{\Box}\|_{L^{2}} \sim \beta$ as $\Box$ varies. Let $\mathbb{S}_{\Box}$ denote the remaining collection of tubes $S$. After pigeonholing one more time we can assume that for each $\Box$ $$ \max_{\substack{ T_r \subset \Box: r \geq K_{1}^2 } } \frac{ \# \{S \in \mathbb{S}_{\Box} : S\subset T_r\}  }{r^{\alpha}} \sim \gamma_1, $$ where $T_r$ are $Kr \times \dotsm \times Kr \times K^2 r$ tubes in $\Box = \Box_{\tau, D}$ running parallel to $G(\tau)$. Finally let $Y_{\Box}$ denote the narrow cubes contained in $\Box$. Sort the $Q \in \mathcal{N}(k)$ into groups according to the value of the multiplicity $$\# \{ \Box : Q \subset Y_{\Box}  \} \sim \mu.$$ We let $\mathbb{B}$ denote the collection of remaining $\Box$'s. It is straightforward to check as in \cite{DZ} that all this pigeonholing contributes an acceptable loss of $(\log R)^c$ to our estimates.
 
 With all these parameters fixed we apply Proposition \ref{prop:narrowDec} and then H\"{o}lder's inequality to obtain \begin{equation} \label{eq:dec1} \|Ef\|_{L^p (Q)} \leq C_{\epsilon}(\log R)^c  K^{m\left(\frac{1}{2} - \frac{1}{p}\right) + \epsilon^{4} }\mu^{\frac{1}{2} - \frac{1}{p} } \big(\sum_{\Box : Q \subset Y_{\Box}} \|Ef_{\Box}\|^{p}_{L^{p} (w_Q) } \big)^{\frac{1}{p}}.
 \end{equation} Moreover, since $\|Ef\|_{L^p (Q)}$ is dyadically constant as $Q$ varies and the number of $Q$ is $\lessapprox R^{d}$ we can assume that \eqref{eq:dec1} holds with the same pigeonholed parameters for each $Q$. Then applying \eqref{eq:dec1} for each such $Q$ and summing, we arrive at the estimate \begin{equation} \label{eq:dec2} \|Ef\|_{L^p (Y)} \leq C_{\epsilon}(\log R)^c  K^{m\left(\frac{1}{2} - \frac{1}{p}\right) + \epsilon^{4} } \mu^{\frac{1}{2} - \frac{1}{p} }\big(\sum_{\Box } \|Ef_{\Box}\|^{p}_{L^{p} (w_{Y_{\Box}}) } \big)^{\frac{1}{p}}.
 \end{equation} We are at the same point as (3.19) in \cite{DZ}, with the only differences being the additional loss of $K^{m\left(\frac{1}{2} - \frac{1}{p}\right)}$ coming from the weaker decoupling for $\mathbb{H}_{m}^{d-1}$, along with the admissible range of $p$. 
 
 We have the following relationships between our parameters. Each of these estimates is proved in exactly the same way as in the parabolic case (see (3.24) and (3.25) in \cite{DZ}).  
 
 \begin{equation}\label{eq:inductiveParameter} \frac{\mu}{\# \mathbb{B}} \lesssim \frac{ (\log R)^{c} M_1 \eta }{M}, \ \ \ \ \ \ \ \ \eta \lesssim \frac{\gamma K^{\alpha +1}}{\gamma_{1}}   \end{equation}

\noindent For each $\Box$ we will estimate the corresponding term in \eqref{eq:dec2} by parabolic rescaling and the induction hypothesis. Since the Fourier transform of $Ef_{\Box}$ is supported in $\tau$, we may find a function $g_{\Box}$ such that $\|g_{\Box}\|_{L^2} = \|f_{\Box}\|_{L^{2}}$ and \begin{equation}\label{eq:rescale} \|Ef_{\Box} \|_{L^{p}(w_{Y_{\Box} }   )} = K^{\frac{d+1}{p} - \frac{d-1}{2} } \|Eg_{\Box}\|_{L^{p}(w_{\widetilde{Y} } ) }, \end{equation} where $\widetilde{Y}$ is the image of $Y_{\Box}$ under the parabolic rescaling. Note that by construction $\widetilde{Y}$ is a union of $K_{1}^2$-cubes which are the images of the pigeonholed $S$ under the rescaling, and these cubes are contained in a ball of radius $R_1$. Moreover, the hypothesis of Proposition \ref{prop:mainInductive} is satisfied for $\widetilde{Y}$ at scale $R_{1}$, with $\# \widetilde{Y} = M_1$ and $\gamma_1$ playing the role of $\gamma(\widetilde{Y})$. Applying \eqref{eq:rescale} and then the induction hypothesis for each $\Box$ to \eqref{eq:dec2} we obtain
\begin{align*}
\|Ef\|_{L^p (Y)} \leq C_{\epsilon}(\log R)^c & K^{m\left(\frac{1}{2} - \frac{1}{p}\right) + \epsilon^{4} } \mu^{\frac{1}{2} - \frac{1}{p} }R^{s(\alpha) + \epsilon} K^{-2s(\alpha) - \epsilon} K^{\frac{d+1}{p} - \frac{d-1}{2}} (\gamma_1 M_{1}^{-1})^{ \frac{1}{2} - \frac{1}{p} } \\ &\cdot  \big(\sum_{\Box }  \|f_{\Box}\|^{p}_{L^{2}}  \big)^{\frac{1}{p}}.
\end{align*}

\noindent Then as a consequence of \eqref{eq:inductiveParameter} we obtain 

\begin{align}\label{eq:inductiveEst} \nonumber
\|E&f\|_{L^p (Y)} \leq C_{\epsilon}(\log R)^c  K^{m\left(\frac{1}{2} - \frac{1}{p}\right) + \epsilon^{4} } R^{s(\alpha) + \epsilon} K^{-2s(\alpha) - \epsilon} K^{\frac{d+1}{p} - \frac{d-1}{2}} (\eta  \gamma_1 M^{-1}\#\mathbb{B} )^{ \frac{1}{2} - \frac{1}{p} } \\ \nonumber& \hspace{50mm}  \cdot \big(\sum_{\Box }  \|f_{\Box}\|^{p}_{L^{2}}  \big)^{\frac{1}{p}}
\\ \nonumber &\leq C_{\epsilon}(\log R)^c  K^{m\left(\frac{1}{2} - \frac{1}{p}\right) + \epsilon^{4} } R^{s(\alpha) + \epsilon} K^{\frac{d+1}{p} - \frac{d-1}{2} -2s(\alpha) - \epsilon} (K^{\alpha +1}\gamma M^{-1})^{\frac{1}{2} - \frac{1}{p}} (\#\mathbb{B} )^{ \frac{1}{2} - \frac{1}{p} }  \\ & \hspace{50mm} \cdot \big(\sum_{\Box }  \|f_{\Box}\|^{p}_{L^{2}}  \big)^{\frac{1}{p}}. 
\end{align}

\noindent Finally, since we are assuming $\|f_{\Box}\|_{L^2} \sim \beta$ for each $\Box \in \mathbb{B}$ it follows that $$(\#\mathbb{B} )^{ \frac{1}{2} - \frac{1}{p} }  \big(\sum_{\Box }  \|f_{\Box}\|^{p}_{L^{2}}  \big)^{\frac{1}{p}} \lesssim \|f\|_{L^{2}}.$$ This completes the sketch of the proof of Lemma \ref{lem:narrowEst}.

\subsection{The broad case} We now suppose that all $Q$ are $(k+1)$-broad. Let $c_Q$ denote the center of each $Q$. By using the uncertainty principle as in the proof of the broad case in \cite{DZ} we arrive at an estimate of the form $$\|Ef\|^{p}_{L^p (Q)} \leq K^{O(1)} \int_{B(c_Q, 2)} \prod_{j=1}^{k+1} |Ef_{j}|^{\frac{p}{k+1}}, \ \ \ \ \ \ Q \in \mathcal{B}(k).$$ Here the $f_j$ are suitable modulations of $f_{\tau_{j}}$'s with $(k+1)$-transverse frequency supports. In particular $\|f_j\|_{L^2} \leq \|f\|_{L^2}$. 

We may pigeonhole to assume that $$\|\prod_{j=1}^{k+1} |Ef_{j}|^{\frac{1}{k+1}} \|_{L^{\infty}(B(c_Q, 2 ))} \sim A$$ for each $Q$. We also fix $q < p$. Then using the above estimates and Bernstein's inequality we obtain \begin{align*} \|Ef\|_{L^p (Y)} &\leq K^{O(1)} \|\prod_{j=1}^{k+1} |Ef|^{\frac{1}{k+1}}\|_{L^{p} (\bigcup_{Q} B(c_Q , 2) )} \\ &\leq K^{O(1)} M^{\frac{1}{p} - \frac{1}{q} } AM^{\frac{1}{q}} \\ &\leq K^{O(1)}  M^{\frac{1}{p} - \frac{1}{q} } \|\prod_{j=1}^{k+1} |Ef_j |^{\frac{1}{k+1}} \|_{L^q (\bigcup_{Q} B(c_Q , 2) ) }.
\end{align*} 

\noindent We now use the $(k+1)$-linear multilinear restriction in $\R^{d}$. Recall the following:

\begin{theorem}[\cite{BCT}] \label{thm:BCT} Suppose $f_{i} \in L^{2}(B^{d-1}(0,2))$ with $f_{i}$ supported in $U_i$ for $i = 1,2,\dotsc,a$ and $a \leq d.$ Also suppose that $U_{i}$ are $a$-transverse, in the sense that $$\inf_{v_i \in U_{i}} |G(v_1) \wedge G(v_2) \wedge \dotsb \wedge G(v_{a})| \gtrsim 1.$$ Then for $q \geq \frac{2a}{a-1}$ and any $\epsilon > 0$ $$\| \prod_{i=1}^{a} |Ef_i|^{\frac{1}{a}} \|_{L^{q}(B^{d}(0,R)) } \leq C_{\epsilon} R^{\epsilon} \prod_{i=1}^{a} \|f_{i}\|^{\frac{1}{a}}_{L^2}.$$
\end{theorem}

\noindent Applying this theorem with $a = k+1$ then yields $$\|Ef\|_{L^p (Y)} \leq C_{\epsilon}R^{\epsilon^2}K^{O(1)}M^{-\left(\frac{1}{2} - \frac{1}{p}\right)} M^{\frac{1}{2} - \frac{1}{q}} \|f\|_{L^2}.$$ We may assume $\gamma \geq K^{-c}$. Then $$M^{\frac{1}{2} - \frac{1}{q}} \leq (\gamma R^{\alpha})^{\frac{1}{2} - \frac{1}{q}} =  \gamma^{\frac{1}{2} - \frac{1}{q}} (R^{\alpha})^{\frac{1}{2(k+1)}} \leq K^{O(1)} \gamma^{\frac{1}{2} - \frac{1}{p}}R^{\frac{\alpha}{2(k+1) }}.$$ If $\delta$ is chosen small enough (depending on $\epsilon$) we therefore obtain \begin{equation}\label{eq:mainBroadEst} \|Ef\|_{L^{p}(Y)} \leq C_{\epsilon} R^{\epsilon} \gamma^{\frac{1}{2} - \frac{1}{p} }M^{-( \frac{1}{2} - \frac{1}{p} )} R^{\frac{\alpha}{2(k+1)}}\|f\|_{L^2}. \end{equation} 

\noindent This completes the proof of Proposition \ref{prop:mainInductive}. 

\subsection{Optimizing between the broad and narrow case}\label{sec:broad} As above we assume $\alpha < d-m$. Note that we get a better estimate for $s_{d}(\alpha)$ in the narrow case when $p = p_{m,k}$ is as large as possible, and hence $k$ is as small as possible (see \eqref{eq:narrowInduction}). However, in the broad case we get a better estimate for $s_{d}(\alpha)$ when $k+1$ is as large as possible, and hence $k$ is as large as possible (see \eqref{eq:mainBroadEst}). The goal now is to optimize between these two cases. From Proposition \ref{prop:mainInductive} we have 
\begin{equation} \label{sbound} s_{d}(\alpha) \leq \max\bigg(\frac{(d - m) - \alpha}{2p_{m,k}} + \frac{\alpha - (d-m) +2 }{4}, \ \frac{\alpha}{2(k+1)} \bigg), \end{equation}
where $$ p_{m,k} = \frac{2(k -m + 1)}{k - m - 1}, \ \ \ \ \  k = m+1, m+2, \dotsc, d-1.$$ Note that $$\frac{(d - m) - \alpha}{2p_{m,k}} + \frac{\alpha - (d-m) +2 }{4} = \frac{k+ 1 - d + \alpha}{2(k-m + 1)},$$ so that \eqref{sbound} becomes 
\begin{equation} \label{sbound2} s_{d}(\alpha) \leq \max\bigg(\frac{k+ 1 - d + \alpha}{2(k-m + 1)}, \ \frac{\alpha}{2(k+1)} \bigg). \end{equation}
Assume first that $m >1$. Write $k+1 = d-j$ and let $j \in [1,m-1]$ be the unique positive integer such that
\[ \alpha \in \left[ \frac{ (j-1)(d-j)}{m-1}, \frac{ j(d-(j+1))}{m-1} \right). \]
This is possible since the intervals on the right hand side form a left-to-right partition of $[0,d-m)$ as $j$ varies from $1$ to $m-1$. This choice of $j$ was found by some tedious algebra which we omit since it is not necessary to the proof. The bound in \eqref{sbound2} becomes 
\begin{align*} s_{d}(\alpha) &\leq \max\bigg(\frac{\alpha-j}{2(d-j-m)}, \ \frac{\alpha}{2(d-j)} \bigg) \\
&= \begin{cases} \frac{\alpha}{2(d-j)}, &\alpha \in \left[ \frac{(j-1)(d-j)}{m-1}, \frac{j(d-j)}{m} \right) \\
\frac{\alpha-j}{2(d-j-m)}, &\alpha \in \left[ \frac{j(d-j)}{m}, \frac{j(d-(j+1))}{m-1} \right), \end{cases} \end{align*} 
and therefore 
\begin{equation}\label{eq:mainBetabound} \beta(\alpha, \mathbb{H}_{m-1}^d ) \geq	\begin{cases} \frac{\alpha(d-j-1)}{d-j}, &\alpha \in \left[ \frac{(j-1)(d-j)}{m-1}, \frac{j(d-j)}{m} \right) \\
\alpha - \frac{\alpha-j}{d-j-m}, &\alpha \in \left[ \frac{j(d-j)}{m}, \frac{j(d-(j+1))}{m-1} \right). \end{cases} 
\end{equation}
This finishes the proof if $m >1$. The simpler case $m=1$ can be handled by making the choice $k= d-2$. 




\subsection{\texorpdfstring{$k$}{k}-broad estimates and failure of transverse equidistribution}

In the argument above we have used the $k$-linear restriction estimates proved in \cite{BCT}. These estimates have nothing to do with the curvature of $\mathbb{H}_{m}^{d-1}$ since they only depend on the transversality of the support of the input functions. Any improvement over the $k$-linear Bennett-Carbery-Tao estimate that takes into account the curvature of $\mathbb{H}_{m}^{d-1}$ will lead to improved lower bounds for $\beta(\alpha, \mathbb{H}^{d-1}_m)$, at least if $k$ is large enough.   

One possible route towards such an improvement would be proving analogues of Guth-type $k$-broad estimates as in \cite{G}. These estimates for the paraboloid are weaker than the conjectured $k$-linear restriction estimates, but still strong enough to yield improved estimates for the extension operator after applying a broad-narrow argument.   

We recall the basic set-up from \cite{G} for $k$-broad norms. On each ball $B_{K^2} \subset B_R$ Guth defines $$\mu_{Ef}(B_{K^2}) := \min_{V_1, \dotsc ,V_A} \big( \max_{\tau \notin V_a} \int_{B_{K^2}} |Ef|^p  \big), $$ where the minimum is over $(k-1)$-dimensional subspaces of $\R^d$ and the maximum is over $\tau$ such that Angle($G(\tau), V_a) \geq \frac{1}{100dK}$ for all $a$ (abbreviated by `$\tau \notin V_a$'). Then the $k$-$broad$ norm (which is not actually a norm) is defined to be $$\|Ef\|^{p}_{\text{BL}^{p}_{k, A}}  := \sum_{B_{K^2} \subset B_R}  \mu_{Ef}(B_{K^2}).$$ For the paraboloid Guth proves that \begin{equation} \label{eq:kBroad} \|Ef\|_{\text{BL}^{p}_{k, A}}  \lesssim_{\epsilon, A} R^{\epsilon}\|f\|_{L^2}\end{equation} for $p$ in exactly the same range as the conjectured sharp $k$-linear restriction bounds, and uses this to deduce new  $L^{p}$ bounds for $Ef$. A key tool used in Guth's argument is a \textit{transverse equidistribution estimate}, which is a certain manifestation of the uncertainty principle when $Ef$ is concentrated on a neighborhood of a lower-dimensional variety. In particular, if $Z$ is an $l$-dimensional variety in $\R^d$ then Guth's transverse equidistribution estimate says that for $\rho < R$ we have \begin{equation}\label{eq:transverseEqui}\int_{N_{\rho^{1 + \delta}} (Z) \cap B_R} |Ef|^{2} \lesssim_{\epsilon} R^{\epsilon} \bigg(\frac{\rho}{R}\bigg)^{d-l} \int_{2B_R} |Ef|^2  \ + \text{RapDec}(R)\|f\|_{L^2}. \end{equation} So the operator cannot concentrate too much on a small neighborhood of $Z$. The curvature of the paraboloid plays an essential role in the proof of \eqref{eq:transverseEqui}, and indeed \eqref{eq:transverseEqui} can fail for hyperboloids for certain $Z$. We explain why below using a slight elaboration on an example found in \cite{GHI} (see example 8.8). This creates an obstacle towards proving $k$-broad estimates for $\mathbb{H}_{m}^{d-1}$, since the estimates \eqref{eq:transverseEqui} play an important role in closing the induction in the proof of \eqref{eq:kBroad} for the paraboloid.

Recently Hickman and Iliopoulou have proved weakened versions of the transverse equidistribution estimates for $\mathbb{H}_{m}^{d-1}$. These imply improved $k$-broad estimates for certain $k$, depending on the signature of the surface (\cite{HiI}). The idea is that even though \eqref{eq:transverseEqui} can fail in general, as long as the dimension of $Z$ is small enough relative to the signature of the surface it is possible to obtained weaked versions of \eqref{eq:transverseEqui} with a smaller power of $\rho/R$. Unfortunately the results in \cite{HiI} do not imply improved $k$-broad estimates in the range we have used in Section \ref{sec:broad} ($k \geq d-m$). In this range the estimates in \cite{HiI} are essentially the same as Theorem \ref{thm:BCT} (for our purposes), and one can check that the $k$-broad estimates in the range $k < d-m$ do not improve the results we have obtained above. Indeed, Theorem 1.5 in \cite{HiI} allows us to replace $\frac{\alpha}{2(k+1)}$ in \eqref{sbound2} by $\frac{\alpha}{d - m + k +1}$ when $m \leq k \leq d-m - 1,$ after a suitable modification of the argument in the broad case. However when $k \leq d-m-1$ we have $\frac{\alpha}{d - m + k +1} \geq \frac{\alpha}{2(d-m)},$ and all the estmates obtained above were at least this good. 

\subsubsection{The example}
We consider the special case $d = 4, m = 1$. After changing coordinates we may assume without loss of generality that $$Ef(x,t) = \int_{\R^3} f(\omega)e^{2\pi i( x\cdot \omega + t(\omega_1 \omega_2 + \omega_{3}^2))} d\omega.$$ As in \cite{G} we will work with a scale $R^{\frac{1}{2}}$ wave packet decomposition 
\[ Ef = \sum_{\theta, \nu} Ef_{\theta, \nu}, \] 
(see Section \ref{preliminaries}).

Below we will identify $\R^3$ with all tuples of the form $(\omega_1, \omega_2, \omega_3, 0).$ If $v = (v_1, v_2, v_3, v_4) \in \R^4$ we also let $\bar{v}$ denote the projection of $v$ onto $\R^3$. Now let $V$ be a three-dimensional subspace of $\R^{4}$ determined by the unit normal vector $$\bar{n} = (n_1, n_2, n_3, 0).$$ We pick $g \in L^{2}(B(0,2))$ such that $Eg$ is essentially concentrated along $V$ in the following sense. Let $\mathcal{V}$ denote the collection of wave packets $T_{\theta, \nu}$ such that$$\text{Angle}(T_{\theta, \nu }, V) < R^{-\frac{1}{2}} \ \ \  \text{and } \ \ T_{\theta, \nu} \subset N_{cR^{\frac{1}{2}}}(V) $$ whenever $T_{\theta, \nu} \in \mathcal{V}.$ Then $$\sum_{ (\theta, \nu) \notin \mathcal{V}}  Eg_{\theta, \nu} = \text{RapDec}(R)\|g\|_{L^2}.$$ 
 
 \noindent Recall that for our operator if $\omega$ is the center of $\theta$ then $T_{\theta, \nu}$ points in the direction $$G_0(\omega) = (-\omega_2, -\omega_1, -2\omega_3, 1).$$ Also note that $g$ must be essentially supported in an $O(R^{-\frac{1}{2}})$ neighborhood of the affine space $$\Omega(V) = \{\omega \in \R^3 : G_0 (\omega) \in V \}.$$ If $$\bar{n}_{\Omega} = (n_2, n_1, 2n_3, 0)$$ then one easily calculates that $$\Omega(V) = \{\omega \in \R^3 : \omega \cdot \bar{n}_{\Omega} = 0 \}.$$  In particular $\Omega(V)$ is a vector space, and we can assume without loss of generality that $|\omega \cdot \bar{n}_{\Omega}| \lesssim R^{-\frac{1}{2}}$ in the support of $g$. 
 
 We now assume that $V$ has been chosen so that $$n_1 n_2 + n_{3}^2 = 0.$$ It follows that $\bar{n}_{\Omega} \in V$ and $\bar{n} \in \Omega(V),$ which we will see is the main obstacle towards proving \eqref{eq:transverseEqui}. Fix a ball $B = B^{4}(0, R^{\frac{1}{2}})$ and a parameter $\rho \in [R^{\frac{1}{2}}, R]$. We wish to prove a lower bound on $$\int_{B \cap N_{c\rho^{\frac{1}{2}}}(V)} |Eg|^2 .$$ Let $$u = \bar{n} \times \bar{n}_{\Omega},$$ so that $u \in \Omega(V)$ and $\{\bar{n}, u,  \bar{n}_{\Omega}  \}$ is an orthonormal basis for $\R^3$. We choose coordinates $\xi$ such that $$\omega = \xi_1 \bar{n} + \xi_2 u + \xi_{3} \bar{n}_{\Omega}.$$ In particular if $A$ is the orthonormal matrix $$A = \begin{bmatrix} \bar{n} & u & \bar{n}_{\Omega} \end{bmatrix}$$ then $\omega = A\xi.$ Let $g_{A}(\xi) = g(A\xi)$ and choose $g$ such that $g_{A}(\xi_1, \xi_{2}, \xi_3 ) = g_{A}(\xi_1, \xi_2)$ for $|\xi_3| \leq R^{-\frac{1}{2}}$ and $g_{A}(\xi_1, \xi_2, \xi_3)  = 0$ for $|\xi_3| > R^{-\frac{1}{2}}.$
 
 Define a linear operator $S$ on $\R^3$ by $$S(v_1, v_2, v_3, 0) = (v_2, v_1, 2v_3, 0).$$ Then for any $v \in \R^3$ one has $$v_1 v_2 + v_3^2 = \frac{1}{2} v \cdot Sv.$$ One can now check that $$\int_{B \cap N_{c\rho^{\frac{1}{2}}}(V)} |Eg(x, t)|^2  = \int_{B \cap N_{c\rho^{\frac{1}{2}}}(A^{T}V)} |\widetilde{E}g_{A} (y, t)|^2, \ \ \ \ y = A^{T}x,$$ where $\widetilde{E}f$ be the operator $$\widetilde{E}f(x,t) = \int_{\R^3} f(\xi) e^{2\pi i (x \cdot \xi + \frac{1}{2}t(\xi_{2}^2 (u \cdot Su) + \xi_1 \xi_3 (\bar{n} \cdot S\bar{n}_{\Omega} + |\bar{n}_{\Omega}|^2) + b)) }d\xi$$ and $b = O(R^{-\frac{1}{2}})$ on the support of $g$ and is independent of $\xi_1$.  
 
 Note that $\vec{e}_1$ is normal to $A^{T}V$. We now let $\phi(\xi_1)$ be a bump function with bounded support and choose $g$ such that $g_{A}(\xi_1, \xi_2, \xi_3) = \phi(\rho^{\frac{1}{2} - \delta} \xi_1 )\phi(\xi_2)$ when $|\xi_3| \lesssim R^{-\frac{1}{2}}$. Since $$\xi_3 (\bar{n} \cdot S\bar{n}_{\Omega} + |\bar{n}_{\Omega}|^2) = O(R^{-\frac{1}{2}})$$ a standard stationary phase argument shows that $\widetilde{E}g_A$ rapidly decays if $|y_1| \geq \rho^{\frac{1}{2}}.$ Since $\vec{e}_1$ is normal to $A^{T}V$ it follows that \begin{align*}\int_{B \cap N_{c\rho^{\frac{1}{2}}}(A^{T}V)} |\widetilde{E}g_{A} (y, t)|^2 &= \int_{B \cap N_{cR^{\frac{1}{2}}}(A^{T}V)} |\widetilde{E}g_{A} (y, t)|^2 + \text{RapDec}(R)\|g\|_{L^2} \\ &=\int_{B \cap N_{cR^{\frac{1}{2}}}(V)} |Eg (x, t)|^2 + \text{RapDec}(R)\|g\|_{L^2} ,\end{align*} and therefore  $$\int_{B \cap N_{c\rho^{\frac{1}{2}}}(V)} |Eg|^2 = \int_{B \cap N_{cR^{\frac{1}{2}}}(V)} |Eg|^2 + \text{RapDec}(R)\|g\|_{L^2}$$ despite the fact that $Eg$ is essentially tangent to $V$. It follows that there can be no transverse equidistribution estimate on any neighborhood of $V$. 
 
 %

 In the next section we give a different lower bound using a refined bilinear argument that takes the weight $\mu$ into account. In the example case $d=4$ and $m=1$ this argument avoids trilinear and 3-broad estimates for $\mathbb{H}_{1}^{3}$, and the bound obtained is sharp when $\alpha \in (2,3) = (d/2, d-m)$.

\section{A bilinear argument in the case \texorpdfstring{$m < \frac{d-1}{2}$}{m< (d-1)/2}}
\label{five}
We now give different lower bounds for $\beta(\alpha, \mathbb{H}^{d-1}_m)$ using a bilinear method similar to the approach to weighted restriction estimates for the paraboloid in \cite{E2}. New ideas are needed, however, to deal with the more complicated transversality assumption one needs to assume to have good bilinear estimates for $\mathbb{H}_{m}^{d-1}$. We need the following two preliminary theorems, the first of which is due to Lee in dimension $d \geq 3$ and Vargas independently in dimension 3 (see \cite{L} and \cite{V}).

\begin{theorem} \label{wolffdecomp} Fix $R \geq 1$ and $\delta >0$. Let $\{\Box\}$ be a finitely overlapping cover of $[-1,1]^d$ by cubes $\Box$ of side length $R^{-\delta}$. Suppose that $F$ and $G$ are such that $\widehat{F}$ and $\widehat{G}$ each have Fourier transform supported in $\mathcal{N}_1(R\mathbb{H}_{m}^{d-1})$, and suppose the supports of $\widehat{F}$ and $\widehat{G}$ are such that if
	\[ \tau_1 := \frac{1}{R} \proj_{\mathbb{R}^{d-1}}\supp \widehat{F}, \quad \tau_2 := \frac{1}{R} \proj_{\mathbb{R}^{d-1}}\supp\widehat{G}, \]
	then 
	\[ \inf_{\substack{ \xi, \overline{\xi} \in \tau_1 \\ 
			\eta, \overline{\eta} \in \tau_2}} \left\lvert \langle (\xi- \eta), M(\overline{\xi} - \overline{\eta} ) \rangle \right\rvert \gtrsim 1. \] Then for each $\Box$, $F$ and $G$ can be decomposed as 
	\[ F = F_{\Box} + F_{\Box^c}, \quad G = G_{\Box} + G_{\Box^c}, \] 
	with $F_{\Box}, F_{\Box^c}, G_{\Box}, G_{\Box^c}$ all supported in $\mathcal{N}_{C}(R\mathbb{H}_{m}^{d-1})$ for some absolute constant $C$, such that for $\epsilon \ll \delta$, 
	\[ \sum_{\Box} \left\lVert F_{\Box} \right\rVert_2^2 \lesssim_{\epsilon} R^{\epsilon} \|F\|_2^2, \quad \sum_{\Box} \left\lVert G_{\Box} \right\rVert_2^2 \lesssim_{\epsilon} R^{\epsilon} \|G\|_2^2, \]
	and 
	\[ \left\lVert \widehat{F_{\Box}} \widehat{G_{\Box^c} } \right\rVert_{L^2(\Box)}, \quad  \left\lVert \widehat{F_{\Box^c}} \widehat{G_{\Box} } \right\rVert_{L^2(\Box)}, \quad  \left\lVert \widehat{F_{\Box^c}} \widehat{G_{\Box^c} } \right\rVert_{L^2(\Box)} \lesssim_{\epsilon,\delta} R^{\frac{d-2}{4}+c\delta}\lVert F\rVert_2 \lVert G \rVert_2, \]
	where the constant $c$ is independent of $\delta$ and $\epsilon$. 
\end{theorem}
\noindent  This is proved in Section 2 of the paper \cite{L} by Lee.  

\begin{theorem} \label{bilinearestimate} Fix $R \geq 1$. Suppose that $F$ and $G$ are supported in $\mathcal{N}_1(R\mathbb{H}_{m}^{d-1})$, and suppose the supports of $F$ and $G$ are are such that if
	\[ \tau_1 := \frac{1}{R} \proj_{\mathbb{R}^{d-1}}\supp \widehat{F}, \quad \tau_2 := \frac{1}{R} \proj_{\mathbb{R}^{d-1}}\supp \widehat{G}, \]
	then 
	\[ \inf_{\substack{ \xi, \overline{\xi} \in \tau_1 \\ 
			\eta, \overline{\eta} \in \tau_2}} \left\lvert \langle (\xi- \eta), M(\overline{\xi} - \overline{\eta} ) \rangle \right\rvert \gtrsim 1. \]  Then for any measure $\mu$ supported on the unit ball and any $q \in [2,4]$,
	\[ \left( \int \left\lvert \widehat{F} \widehat{G} \right\rvert^{q/2} \, d\mu \right)^{2/q} \lesssim_a R^{2a} c_{\alpha}(\mu)^{2/q} \lVert F\rVert_2 \lVert G\rVert_2, \]
	for any 
	\[ a > \max\left\{ \frac{d-1}{2}- \frac{\alpha}{q}, \frac{3d}{8} - \frac{(\alpha+1)}{4} \right\}. \]
\end{theorem}
\begin{proof} The argument is similar to the proof of the weighted bilinear estimates for the paraboloid and cone in \cite{E2} and \cite{ChHaLe1}. By induction it may be assumed that the result holds at scales smaller than $R/2$. Fix a small $\delta >0$ and break the unit ball into cubes $\Box$ of side length $R^{-\delta}$. Since $q \in [2,4]$,
	\begin{align} \notag  \left\lVert  \widehat{F} \widehat{G} \right\rVert_{L^{q/2}(\mu) }  &\leq \sum_{\Box} \left\lVert  \widehat{F} \widehat{G} \right\rVert_{L^{q/2}(\mu, \Box/2)} \\
	\notag &\lesssim \sum_{\Box} \left\lVert  \widehat{F_{\Box}} \widehat{G_{\Box}} \right\rVert_{L^{q/2}(\mu, \Box)}+ \left\lVert  \widehat{F_{\Box}} \widehat{G_{\Box_c}} \right\rVert_{L^{q/2}(\mu, \Box/2)} \\
	\label{fourparts} &\quad +\left\lVert  \widehat{F_{\Box_c}} \widehat{G_{\Box}} \right\rVert_{L^{q/2}(\mu, \Box/2)}+ \left\lVert  \widehat{F_{\Box_c}} \widehat{G_{\Box_c}} \right\rVert_{L^{q/2}(\mu, \Box/2)}.  \end{align} 
	Let $\rho = R^{\delta}$. For each $\Box$, let $\mu_{\rho} = \rho_{\#}(\mu \chi_{\Box} )$, which is supported in a ball of radius $\sim 1$. The first term is
	\begin{align} \notag  \left\lVert  \widehat{F_{\Box}} \widehat{G_{\Box}} \right\rVert_{L^{q/2}(\mu, \Box)} &= \left( \int \left\lvert \widehat{F}(x/\rho) \widehat{G}(x/\rho) \right\rvert^{q/2} \, d\mu_{\rho}(x) \right)^{2/q} \\
	\label{localise} &= \left( \int \left\lvert \widehat{F_{\rho}}(x) \widehat{G_{\rho}}(x) \right\rvert^{q/2} \, d\mu_{\rho}(x) \right)^{2/q}, \end{align}
	where $F_{\rho}$ is defined by $\widehat{F_{\rho}}(x) = \widehat{F_{\Box}}(x/ \rho)$. Then $F_{\rho}$ and $G_{\rho}$ are supported in $\mathcal{N}_{\rho^{-1}}(R \rho^{-1}\mathbb{H})$.  Let $\phi$ be a Schwartz function such that $\lvert \phi \rvert \sim 1$ in a ball of radius $\sim 1$ containing the support of $\mu_{\rho}$, and such that $\widecheck{\phi}$ is compactly supported in a ball around the origin. Then applying the induction hypothesis at scale $R/\rho$ yields  
	\begin{align} \notag \eqref{localise} &\lesssim \left( \int \left\lvert \widehat{\widecheck{\phi} \ast F_{\rho}}(x) \widehat{\widecheck{\phi} \ast G_{\rho}}(x) \right\rvert^{q/2} \, d\mu_{\rho}(x) \right)^{2/q} \\
	\label{improve} &\lesssim \left( \frac{R}{\rho} \right)^{2a} c_{\alpha}(\mu_{\rho})^{2/q} \left\lVert \widecheck{\phi} \ast F_{\rho} \right\rVert_2 \left\lVert \widecheck{\phi} \ast G_{\rho} \right\rVert_2. \end{align} 
	By Fubini,
	\[  \left\lVert \widecheck{\phi} \ast F_{\rho} \right\rVert_1 \lesssim \left\lVert F_{\rho} \right\rVert_1. \]
	The measure of a ball of radius 1 intersected with the $\rho^{-1}$-neighbourhood of the hyperboloid is $\lesssim \rho^{-1}$, so 
	\[  \left\lVert \widecheck{\phi} \ast F_{\rho} \right\rVert_{\infty} \lesssim \rho^{-1} \left\lVert F_{\rho} \right\rVert_{\infty}. \]
	By interpolation, 
	\[  \left\lVert \widecheck{\phi} \ast F_{\rho} \right\rVert_2 \lesssim \rho^{-1/2} \left\lVert F_{\rho} \right\rVert_2 = \rho^{\frac{d-1}{2}} \left\lVert F_{\Box} \right\rVert_2. \]
	Using this and the inequality $c_{\alpha}(\mu_{\rho}) \lesssim \rho^{-\alpha}c_{\alpha}(\mu)$ gives  
	\begin{align*} \eqref{improve} &\lesssim R^{2a} c_{\alpha}(\mu)^{2/q} \rho^{-2a+d-1-\frac{2\alpha}{q}} \left\lVert F_{\Box}\right\rVert_2 \left\lVert G_{\Box}\right\rVert_2 \\
	&\ll R^{2a} c_{\alpha}(\mu)^{2/q} \left\lVert F_{\Box}\right\rVert_2 \left\lVert G_{\Box}\right\rVert_2, \end{align*}
	by the assumption on $a$. 
	
	For the first off-diagonal term, let $\phi_R$ be a smooth bump function equal to $1$ on $B(0,10CR)$ and vanishing outside a ball with the same centre and twice the radius (here $C$ is the same constant from Theorem \ref{wolffdecomp}). By H\"{o}lder's inequality,
	\[ \left\lvert \widehat{F_{\Box}} \widehat{G_{\Box_c}} \right\rvert^{q/2} = \left\lvert \widehat{F_{\Box}} \widehat{G_{\Box_c}} \ast \widehat{\phi_R} \right\rvert^{q/2} \lesssim \left\lvert \widehat{F_{\Box}} \widehat{G_{\Box_c}}  \right\rvert^{q/2} \ast \left\lvert \widehat{\phi_R} \right\rvert. \]
	Hence,  
	\begin{align*} &\left\lVert  \widehat{F_{\Box}} \widehat{G_{\Box_c}} \right\rVert_{L^{q/2}(\mu, \Box/2)} \\
	&\quad = \left( \int_{\Box/2}  \left\lvert \widehat{F_{\Box}} \widehat{G_{\Box_c}} \right\rvert^{q/2} \, d\mu \right)^{2/q}  \\
	&\quad \lesssim \left( \int_{\Box/2}  \left\lvert \widehat{F_{\Box}} \widehat{G_{\Box_c}} \right\rvert^{q/2} \ast \left\lvert \widehat{\phi_R} \right\rvert \, d\mu  \right)^{2/q}  \\
	&\quad \lesssim_N \left( \int_{\Box}  \left\lvert \widehat{F_{\Box}}(y) \widehat{G_{\Box_c}}(y) \right\rvert^{q/2} \left(\left\lvert \widehat{\phi_R} \right\rvert \ast \mu\right)(y) \, dy  \right)^{2/q} + R^{-N}  \left\lVert F \right\rVert_2 \left\lVert G\right\rVert_2,   \end{align*} 
	for arbitrarily large $N$. By H\"{o}lder's inequality, the first term satisfies 
	\begin{align} \notag  &\left( \int_{\Box}  \left\lvert \widehat{F_{\Box}}(y) \widehat{G_{\Box_c}}(y) \right\rvert^{q/2} \left(\left\lvert \widehat{\phi_R} \right\rvert \ast \mu\right)(y) \, dy  \right)^{2/q} \\
	\label{holder} &\quad \leq \left\lVert \widehat{F_{\Box}} \widehat{G_{\Box^c} } \right\rVert_{L^2(\Box)} \left\lVert \left\lvert \widehat{\phi_R} \right\rvert \ast \mu \right\rVert_{\frac{4}{4-q} }^{2/q}.  \end{align} 
	The function $\left\lvert \widehat{\phi_R} \right\rvert \ast \mu$ is essentially supported in a ball of radius $\sim 1$, has $L^{\infty}$ norm bounded by $R^{d-\alpha+O(\delta)} c_{\alpha}(\mu)$, and has $L^1$ norm $\lesssim c_{\alpha}(\mu)$. Therefore 
	\begin{align*} \left\lVert \left\lvert \widehat{\phi_R} \right\rvert \ast \mu \right\rVert_{\frac{4}{4-q} } &\leq R^{\frac{q}{4} \left( d-\alpha+O(\delta)\right)} c_{\alpha}(\mu).  \end{align*} 
	Combining this with the bound from Theorem \ref{wolffdecomp} for the first term gives
	\[ \eqref{holder} \lesssim R^{\frac{d-2}{4} + c\delta}  R^{\frac{1}{2} \left( d-\alpha+O(\delta)\right)} c_{\alpha}(\mu)^{2/q} \left\lVert F_{\Box}\right\rVert_2 \left\lVert G_{\Box}\right\rVert_2 \ll R^{2a} c_{\alpha}(\mu)^{2/q} \left\lVert F_{\Box}\right\rVert_2 \left\lVert G_{\Box}\right\rVert_2, \]
	where the last inequality comes from the assumption on $a$. The bound on the other three diagonal terms is similar, so putting this into \eqref{fourparts}, applying Cauchy-Schwarz and then Theorem \ref{wolffdecomp} finishes the proof. 
\end{proof} 

\subsection{A bilinear version of the Du-Zhang method}  

The following is the main result of the section, which implies the lower bounds for $\beta(\alpha,\mathbb{H}_{m}^{d-1})$ claimed in \eqref{bilinear}. The argument combines the Du-Zhang approach from \cite{DZ} with some recent observations about Lee's and Vargas' bilinear estimates for the hyperboloid from \cite{Ba}. It can also be adjusted to work when $d$ is odd and $m= \frac{d-1}{2}$, but in that case the inequality obtained is worse than the trivial bound.  

\begin{theorem} Assume that $1 \leq m < \frac{d-1}{2}$. 
Fix $\alpha \in \left( \frac{d-1}{2},\frac{d}{2} + 1\right)$ and $0< \delta \ll \epsilon$ and set $K = R^{\delta}$. Let $Y = \bigcup_{j=1}^M Q_j$ be an $\alpha$-dimensional collection of $K^{2}$-cubes in $B^{d}(0,R)$, with $\alpha$-dimensional constant $\gamma$. Also suppose that $\|Ef\|_{L^{\infty}(Q_j)}$ is dyadically constant over $Q_j \subseteq Y$. Then 
	\[ \|Ef\|_{L^{\infty}(Y)} \lesssim_{\epsilon} \gamma^{1/2} M^{-1/2}R^{\frac{\alpha+1}{4} - \frac{d}{8}+\epsilon} \|f\|_2. \]
\end{theorem} 
\begin{proof} Suppose first that more than half of the cubes in $Y$ are $(m+2)$-broad. Since $\|Ef\|_{L^{\infty}(Q)}$ is dyadically constant as $Q$ varies we have 
	\begin{equation} \label{equidist2} \left\lVert Ef \right\rVert_{L^{\infty}(Y)} \leq M^{- 1/2   } \left( \sum_{Q \subseteq Y } \left\lVert Ef \right\rVert_{L^{\infty}(Q)}^2 \right)^{1/2}. \end{equation}
	
	Then for each $(m+2)$-broad cube $Q$, there exist significant caps $\tau_1, \dotsc, \tau_{m+2}$ at scale $K^{-1}$, depending on $Q$, such that 
	\[ \|Ef\|_{L^{\infty}(Q)} \leq K^{O(1)} \left\lVert Ef_{\tau_i} \right\rVert_{L^{\infty}(Q)}, \]
	for every $i \in \{1,\dotsc,m+2\}$, and such that 
	\begin{equation} \label{transversality} \left\lvert n(\xi_1) \wedge \dotsb \wedge n(\xi_{m+2}) \right\rvert \geq K^{-1}, \quad \text{for all } \xi_i \in \tau_i, \end{equation}
where $n(\xi)$ is the normal to the hyperboloid at $(\xi, \langle \xi, M \xi \rangle)$ and $M$ is the diagonal matrix associated to $\mathbb{H}^{d-1}_m$. For each $i$ let $\tau_i^*$ be a $K^{-10}$-cap inside $\tau_i$ such that 
	\[ \left\lVert Ef_{\tau_i} \right\rVert_{L^{\infty}(Q)} \leq K^{O(1)} \left\lVert Ef_{\tau_i^*} \right\rVert_{L^{\infty}(Q)}. \]
	Suppose for a contradiction that for all $i \neq j$, there exist $\xi \in \tau_i^*$ and $\eta \in \tau_j^*$ with 
	\[ \left\lvert \left\langle M(\xi-\eta), \xi-\eta \right\rangle \right\rvert \leq 100K^{-10}, \]
Since the $\tau_i^*$ are $K^{-10}$-caps, it follows that for all $i \neq j$,
\[ \left\langle M(\xi_i-\xi_j), \xi_i-\xi_j \right\rangle  = O(K^{-10}). \]
	for all $\xi_i \in \tau_i^*$. Fix any such tuple and let $\widetilde{\xi}_i = \xi_i - \xi_{m+2}$ for all $1 \leq i \leq m+1$. Then the previous estimates yield
	\[ \left\langle M(\widetilde{\xi}_i-\widetilde{\xi}_j), \widetilde{\xi}_i-\widetilde{\xi}_j \right\rangle  = O(K^{-10}), \quad 1 \leq i, j \leq m+1, \]
	and
\begin{equation} \label{equidist} \left\langle M\widetilde{\xi}_i, \widetilde{\xi}_i \right\rangle  = O(K^{-10}) , \quad 1 \leq i \leq m+1. \end{equation}
	Since $M$ is self-adjoint, subtraction gives 
	\begin{equation} \label{orthogonal} \left\langle M\widetilde{\xi}_i,  \widetilde{\xi}_j \right\rangle = O(K^{-10}), \quad 1 \leq i,j \leq m+1. \end{equation}
	By \eqref{transversality}, 
	\[ \left\lvert (\xi_1, 1) \wedge \dotsb \wedge (\xi_{m+2}, 1) \right\rvert \gtrsim K^{-1}, \]
	which by the property $v \wedge v = 0$ implies that 
	\begin{equation} \label{placeholder} \left\lvert \widetilde{\xi}_1 \wedge \dotsb \wedge \widetilde{\xi}_m \right\rvert \gtrsim \left\lvert \widetilde{\xi}_1 \wedge \dotsb \wedge \widetilde{\xi}_{m+1} \right\rvert \gtrsim K^{-1}. \end{equation}
	Expand out $\widetilde{\xi}_{m+1}$ in the near-orthogonal basis
	\[ \widetilde{\xi}_{m+1} = v_1 + v_2 + v_3, \]
	with  
	\[ v_1 \in \spn\left\{\widetilde{\xi}_i : 1 \leq i \leq m\right\}, \quad  v_2 \in \spn\left\{M\widetilde{\xi}_i : 1 \leq i \leq m \right\}, \]
	and 
	\[ v_3 \in \spn\left\{\widetilde{\xi}_i, M\widetilde{\xi}_i : 1 \leq i \leq m\right\}^{\perp}. \]
	More specifically, $v_3$ is the orthogonal projection of $\widetilde{\xi}_{m+1}$ onto 
	\[ \spn\left\{\widetilde{\xi}_i, M\widetilde{\xi}_i : 1 \leq i \leq m\right\}^{\perp}, \] and $(v_1+v_2)$ is the orthogonal projection of $\widetilde{\xi}_{m+1}$ onto 
	\[ \spn\left\{\widetilde{\xi}_i, M\widetilde{\xi}_i : 1 \leq i \leq m\right\}. \] Hence there exist constants $\lambda_i$ and $\mu_i$ such that 
	\[ v_1 = \sum_{i=1}^m \lambda_i \widetilde{\xi}_i, \quad v_2 = \sum_{i=1}^m \mu_i M\widetilde{\xi}_i. \]
	By \eqref{placeholder}, the coefficients $\lambda_i$ and $\mu_i$ satisfy 
	\begin{equation} \label{pause17} \left( \sum_{i=1}^m \left\lvert \lambda_i \right\rvert^2 \right)^{1/2} \lesssim K \left\lvert v_1 \right\rvert, \quad \left( \sum_{i=1}^m \left\lvert \mu_i \right\rvert^2 \right)^{1/2} \lesssim K \left\lvert v_2 \right\rvert, \end{equation}
	and so by Cauchy-Schwarz and \eqref{orthogonal},
	\[ \left\lvert \left\langle v_1, v_2 \right\rangle \right\rvert \lesssim K^{-8} \left\lvert v_1 \right\rvert \left\lvert v_2 \right\rvert, \]
	which gives (for $K$ larger than a fixed constant)
	\[ 1 \gtrsim \left\lvert v_1+v_2\right\rvert^2  \gtrsim \left\lvert v_1 \right\rvert^2 + \left\lvert v_2 \right\rvert^2. \]
Hence by \eqref{pause17} the coefficients $\lambda_i$ and $\mu_i$ all have size $\lesssim K$. By \eqref{placeholder}, $\left\lvert \widetilde{\xi}_i \right\rvert, \left\lvert M\widetilde{\xi}_i \right\rvert \gtrsim K^{-1}$. Hence by \eqref{orthogonal},
	\[ |v_2|^2 = \left\langle v_2, \widetilde{\xi}_{m+1} - v_1 \right\rangle \lesssim K^{-8}. \] 
	Similarly, the $v_3$ term satisfies
	\begin{equation} \label{v3bound} \left\lvert \left\langle v_3, Mv_3 \right\rangle\right\rvert = \left\lvert \left\langle \widetilde{\xi}_{m+1} - v_1, M(\widetilde{\xi}_{m+1} - v_1) \right\rangle\right\rvert + O(K^{-4}) = O(K^{-4}). \end{equation}
	The range of $(I-M)$ is $m$-dimensional and contains the $m$ vectors $\widetilde{\xi}_i - M\widetilde{\xi}_i$ for $1 \leq i \leq m$. These $m$ vectors are linearly independent since otherwise the identity 
	\[ w := \widetilde{\xi}_j -\sum_{i \neq j} a_i \widetilde{\xi}_i =  M\widetilde{\xi}_j -\sum_{i \neq j} a_i M\widetilde{\xi}_i, \]
for some $1 \leq j \leq m$ and constants $a_i$ would give, by \eqref{orthogonal} and \eqref{placeholder}, the contradiction
	\[ K^{-1} \sup_j(1, |a_j|) \lesssim |w|  \lesssim \sup_j(1, |a_j|)K^{-5}. \]
Hence $v_3 \in \ran(I-M)^{\perp} = \ker(I-M)$ since $M$ is self-adjoint, and \eqref{v3bound} gives $|v_3| \lesssim K^{-2}$. Hence
	\[ \widetilde{\xi}_{m+1} = v_1 + O(K^{-2}). \]
	Returning to the old coordinates gives 
	\[ \xi_{m+1} = \sum_{i=1}^m \lambda_i \xi_i +\left( 1- \sum_{i=1}^m \lambda_i  \right) \xi_{m+2}  + O(K^{-2}). \]
	Using this and the property $v \wedge v =0$ of the wedge product yields 
\[\left\lvert n(\xi_1) \wedge \dotsb \wedge n(\xi_{m+2}) \right\rvert = O(K^{-2}), \]
	which contradicts \eqref{transversality} (provided $K$ is large enough).  This contradiction proves that there exists a pair $(i,j)$ with $i \neq j$ such that 
	\[ \left\lvert \left\langle M(\xi-\eta), \xi-\eta \right\rangle \right\rvert \geq 100 K^{-10}, \]
	for every $\xi \in \tau_i^*$ and $\eta \in \tau_j^*$ with $i \neq j$. Assume without loss of generality that $i=1$ and $j=2$. Then since the $\tau_1^*, \tau_2^*$ are $K^{-10}$-caps, 
	\begin{equation} \label{bisep} \inf_{\substack{\xi, \overline{\xi} \in \tau_1^* \\
			\omega, \overline{\omega} \in \tau_2^* }} \left\lvert \left\langle M(\xi-\eta), \overline{\xi}-\overline{\eta} \right\rangle \right\rvert \geq K^{-10}. \end{equation}
	For each $Q \subseteq Y$ there are $\lesssim K^{O(1)}$ corresponding pairs of caps $(\tau_1^*, \tau_2^*)$ defined above satisfying \eqref{bisep}, so by pigeonholing \eqref{equidist2} and Bernstein's inequality, there exists a fixed pair $(\tau_1^*, \tau_2^*)$ such that
	\begin{align} \notag \left\lVert Ef \right\rVert_{L^{\infty}(Y)} &\leq M^{-1/2 } K^{O(1)} \left( \sum_{Q \subseteq Y } \left\lVert \left\lvert Ef_{\tau_1^*} Ef_{\tau_2^*} \right\rvert^{1/2} \right\rVert_{L^2(10Q)}^2 \right)^{1/2} \\
	\label{nearlyfinished} &\leq M^{- 1/2 } K^{O(1)} \left\lVert \left\lvert Ef_{\tau_1^*} Ef_{\tau_2^*} \right\rvert^{1/2} \right\rVert_{L^2(10Y)}. \end{align}
	Define $f_R$ so that $Ef_R(x) = Ef(Rx)$. We can now apply Theorem \ref{bilinearestimate} after a minor localisation argument (see for example Section 2.2 in \cite{Ba}). This yields
	\begin{align*} \left\lVert \left\lvert Ef_{\tau_1^*} Ef_{\tau_2^*} \right\rvert^{1/2} \right\rVert_{L^2(10Y)} &= R^{\frac{\alpha}{2}} \left\lVert \left\lvert Ef_{\tau_1^*,R} Ef_{\tau_2^*,R} \right\rvert^{1/2} \right\rVert_{L^2\left(\frac{10Y}{R}, R^{d-\alpha} \, dm \right)} \\
	&\lesssim_a R^{\frac{\alpha}{2}+a} \left\lVert Ef_{\tau_1^*,R} \right\rVert_{L^2(B(0,100))}^{1/2} \left\lVert Ef_{\tau_2^*,R} \right\rVert_{L^2(B(0,100))}^{1/2} \\
	&\lesssim R^{\frac{\alpha}{2}-\frac{d}{2}+a} \left\lVert Ef_{\tau_1^*} \right\rVert_{L^2(B(0,100R))}^{1/2} \left\lVert Ef_{\tau_2^*} \right\rVert_{L^2(B(0,100R))}^{1/2}   \\
	&\lessapprox R^{ \frac{\alpha+1}{4}  - \frac{d}{8}} \|f\|_2,  \end{align*}
where the assumption $\alpha>\frac{d-1}{2}$ was used to take the second term of the max in Theorem~\ref{bilinearestimate}. Putting this into \eqref{nearlyfinished} gives 
	\[ \left\lVert Ef \right\rVert_{L^{\infty}(Y)} \lessapprox  M^{-1/2  } R^{\frac{\alpha+1}{4}  - \frac{d}{8}} \|f\|_2, \]
	which proves the inequality in the broad case.  The narrow case is no different than what we saw in Section \ref{four}. In particular from \eqref{eq:narrowDecEq} and the assumption $m < \frac{d-1}{2}$, the exponent $\frac{\alpha+1}{4}  - \frac{d}{8}$ is at least as large as the best possible that passes through $(m+1)$-narrow ($L^{\infty}$) decoupling (use $k=m+1$ in \eqref{sbound2}), so this finishes the proof. \end{proof}

\begin{cor} Assume that $1 \leq m < \frac{d-1}{2}$ and fix $\alpha \in \left(\frac{d-1}{2},\frac{d}{2} + 1\right)$. Then for any $\alpha$-dimensional set $X$ of unit cubes in $B^d(0,R)$, 
	\[ \|Ef\|_{L^2(X)} \lesssim_{\epsilon} R^{\frac{\alpha+1}{4} - \frac{d}{8}+\epsilon} \|f\|_2, \] 
and so
\[ \beta(\alpha, \mathbb{H}^{d-1}_m) \geq \frac{\alpha}{2} + \frac{d}{4} - \frac{1}{2}, \quad \alpha \in \left(\frac{d-1}{2} ,\frac{d}{2} + 1\right). \]
By Proposition \ref{prop:upperBound} this is sharp if $d \geq 4$ is even, $m = \frac{d}{2} -1$ and $\alpha \in (d-m-1,d-m)$. Consequently \[ \beta(\alpha, \mathbb{H}^{d-1}_m) = \frac{\alpha}{2} + \frac{d}{4} - \frac{1}{2}, \quad \alpha \in \left( d-m-1, d-m\right), \quad d \text{ even and } m = \frac{d}{2} -1. \]  
\end{cor} 

\begin{remark} Since we are working in $L^{\infty}$, our argument does not require the more sophisticated decoupling result in Proposition \ref{prop:narrowDec}. In particular the decoupling estimate we use in this case is a simple consequence of the Cauchy-Schwarz inequality. 
\end{remark}

 \section*{Appendix: Equivalence between decay and localized restriction estimates }
 
 In this appendix we explain how to prove Proposition \ref{prop:equiv} and Lemma \ref{lem:latticeCube}. These results are not new but we include arguments for the convenience of the reader. 
 
\begin{proof}[Proof of Proposition \ref{prop:equiv}] We only show that $s_{d}(\alpha) \leq \frac{\alpha - \beta(\alpha, \mathbb{H}^{d-1}_m) }{2},$ since the other inequality is more well-known and easily follows from duality and Plancherel (see for example \cite{DGOWWZ}). Our argument is essentially the same as in \cite{bbcr,R,H}, but we include the details since we are working with localized $s_{d}(\alpha)$ and the identity is different than in \cite{bbcr,R,H}. There is nothing special about the geometry of $\mathbb{H}_{m}^{d-1}$, and the argument will work for any smooth hypersurface.  
 
 The idea is to use $\|Ef\|_{L^1(d\nu)}$ bounds to control the measure of level sets 
\[ \mu_R \{|Ef| > \lambda \}, \]
 and use these estimates to bound $\|Ef\|_{L^2(\mu_R)}.$ Note that for any Borel measure $\nu$ supported on $B^{d}(0,R)$ we have, by duality, Plancherel, and Cauchy-Schwarz, \begin{equation}\label{eq:weightedL1est}|\int Ef d\nu| \lesssim \|f\|_{L^2 (\mathbb{H}_{m}^{d-1} )} \bigg( \int_{\mathbb{H}_{m}^{d-1}} |\widehat{\nu}(\xi)|^2 d\sigma(\xi) \bigg)^{\frac{1}{2}}.\end{equation} 
 
 \noindent For each $\lambda > 0$ we let $\mu_{R, \lambda}$ be the measure on $B_R$ defined by $$\mu_{R,\lambda}(F) = \frac{ \mu_R(F \cap \{|Ef| > \lambda \} )  }{\mu_R (\{|Ef| > \lambda \})  } = \frac{ \mu(R^{-1} (F \cap \{|Ef| > \lambda \} ))  }{\mu (R^{-1}\{|Ef| > \lambda \})  }.$$ We assume for now that $\mu_R$ is supported on the set where $Ef > 0$. Make the normalization $\|f\|_{L^2} = 1$ and apply \eqref{eq:weightedL1est} with $\nu = \mu_{R,\lambda}$ to obtain $$\int |Ef| d\mu_{R, \lambda} \lesssim \bigg(\int_{\mathbb{H}_{m}^{d-1}} |\widehat{\mu_{R, \lambda}}(\xi)|^2 d\sigma(\xi) \bigg)^{\frac{1}{2}}.$$ 
 
 \noindent Now let $\nu$ be the probability measure $$\nu = \frac{\mu|_{R^{-1} \{|Ef| > \lambda \} } }{ \mu(R^{-1} \{ |Ef| > \lambda \} ) }.$$ Then $$ \widehat{\mu_{R, \lambda}}(\xi) = \int_{\R^d} e^{-2\pi i (Rx\cdot \xi )}  d\nu(x) = \widehat{\nu}(R\xi),$$ and therefore we obtain  \begin{align}\label{eq:betaEst} \int |Ef| d\mu_{R, \lambda} &\leq \bigg(\int_{\mathbb{H}_{m}^{d-1}} |\widehat{\nu}(R\xi)|^2 d\sigma(\xi) \bigg)^{\frac{1}{2}} \\ \nonumber &\lesssim R^{-\frac{\beta}{2}} c_{\alpha}(\nu)^{\frac{1}{2}} \end{align} for any $\beta < \beta(\alpha, \mathbb{H}^{d-1}_m)$. 
 
 We claim that \begin{equation} \label{eq:nuEst} c_{\alpha}(\nu) \lesssim R^{\alpha}c_{\alpha}(\mu_{R}) \mu(R^{-1}\{|Ef| > \lambda \} )^{-1}.
 \end{equation} Indeed, let $B_r$ be a ball of radius $r$ such that $c_{\alpha}(\nu) \sim \frac{\nu(B_r)}{r^{\alpha}}.$ Then \begin{align*} c_{\alpha}(\nu) &\sim \frac{ \mu( B_r \cap R^{-1}\{|Ef| > \lambda \} ) }{r^{\alpha}} \mu(R^{-1}\{|Ef| > \lambda \})^{-1} \\ &= \frac{ \mu(R^{-1}( B_{Rr} \cap \{|Ef| > \lambda \} )) }{r^{\alpha}} \mu(R^{-1}\{|Ef| > \lambda \})^{-1} \\ &=R^{\alpha} \frac{\mu_{R, \lambda} (B_{Rr})}{(rR)^{\alpha}} &\\ &\leq R^{\alpha}c_{\alpha}(\mu_{R, \lambda}) \\& \leq R^{\alpha} c_{\alpha}(\mu_R) \mu_R(\{|Ef| > \lambda \})^{-1},  \end{align*} proving \eqref{eq:nuEst}. 
 
 It follows from \eqref{eq:betaEst}, \eqref{eq:nuEst}, and the definition of $\mu_{R,\lambda}$ that $$\lambda \lesssim R^{\frac{\alpha - \beta}{2}} \mu_R(\{|Ef| > \lambda \})^{-\frac{1}{2}}  c_{\alpha}(\mu_R)^{\frac{1}{2}},$$ and therefore \begin{equation}\label{eq:levelSet} \mu_R(\{|Ef| > \lambda \}) \lesssim \lambda^{-2} R^{\alpha - \beta} c_{\alpha}(\mu_R). 
 \end{equation} Recall that $$\int |Ef|^2 d\mu_{R} = 2 \int_{0}^{\infty} \lambda \mu_R (\{|Ef| > \lambda \}) d\lambda = 2 \int_{0}^{C} \lambda \mu_R (\{|Ef| > \lambda \}) d\lambda,$$ the latter identity following from the normalization $\|f\|_{L^2} =1$. Then $$\int_{0}^{R^{-\frac{\beta}{2}}} \lambda \mu_R (\{|Ef| > \lambda\}) d\lambda \leq R^{\alpha-\beta} $$ and as a consequence of \eqref{eq:levelSet} we also obtain \begin{align*}\int_{R^{-\frac{\beta}{2} }}^{C} \lambda \mu_R (\{|Ef| > \lambda\}) d\lambda &\lesssim c_{\alpha}(\mu_R) R^{\alpha - \beta} \int_{R^{-\frac{\beta}{2}}}^{C} \lambda^{-1} d\lambda \\ &\lesssim_{\delta} c_{\alpha}(\mu_R) R^{\alpha - \beta + \delta}  \end{align*} for any $\delta > 0$. 
 
 We have therefore shown that $$\|Ef\|_{L^2 (\mu_R )} \lesssim_{\delta} c_{\alpha}(\mu_R)^{\frac{1}{2}} R^{\frac{\alpha - \beta + \delta}{2}}\|f\|_{L^2}$$ for any $\delta > 0$ and for any $\beta < \beta(\alpha, \mathbb{H}^{d-1}_m)$ when $\mu_R$ is supported in the set where $Ef > 0$. In the general case we decompose $$Ef = g_1 - g_2 + i(g_3 - g_4), \ \ \  \ g_j \geq 0$$ and let $G_{j}$ denote the support of $g_j$. Then we set $\mu_{R}^{j} = \mu_{R}|_{G_j}$ and repeat the above argument to obtain \begin{align*}\|Ef\|_{L^2 (\mu_R )}  &\lesssim_{\delta} \big(\sum_{j=1}^{4} c_{\alpha}(\mu_{R}^{j}) \big)^{\frac{1}{2}}R^{\frac{\alpha - \beta + \delta}{2}}\|f\|_{L^2} \\ &\lesssim_{\delta} c_{\alpha}(\mu_{R})^{\frac{1}{2}}R^{\frac{\alpha - \beta + \delta}{2}}\|f\|_{L^2}.  \end{align*} It follows that $$s_{d}(\alpha) \leq \frac{\alpha - \beta}{2} + \frac{\delta}{2}$$ for any $\delta > 0$ and $\beta < \beta(\alpha, \mathbb{H}^{d-1}_m)$. Letting $\beta \rightarrow \beta(\alpha, \mathbb{H}^{d-1}_m)$ and $\delta \rightarrow 0$ completes the proof. \end{proof} 
 
\begin{proof}[Proof of Lemma \ref{lem:latticeCube}] We first show $s_{d}(\alpha) \geq \overline{s_d}(\alpha).$  Fix $\delta > 0$ and choose $s < s_{d}(\alpha)$ such that for all $\alpha$-dimensional $\mu$ supported in $B^{d}(0,1)$ $$\|Ef\|_{L^2 (\mu_R; B^{d}(0,R))} \lesssim_{\delta} c_{\alpha}(\mu_R)^{\frac{1}{2}} R^{s + \delta} \|f\|_{L^2}$$ whenever supp$(f) \subset B^{d-1}(0,2).$ 

Let $X$ be a union of lattice unit cubes in $B^{d}(0,R)$ such that $\gamma_{\alpha}(X) \sim 1$. We define a measure $\mu$ on $B^{d}(0,1)$ such that $\mu(E\cap R^{-1} Q) = \frac{|RE \cap Q|}{\# X}$ if $Q \in X$ and such that $\mu(L) = 0$ for every other lattice $R^{-1}$-cube $L$. Then $\mu$ is a probability measure. For any ball $B_r$ with $r \geq R^{-1}$ we have $$\frac{\mu(B_r)}{r^{\alpha}} \lesssim \frac{1}{\# X} \frac{\#\{Q \in X : Q \subset B_{cRr} \}}{r^{\alpha}} = \frac{R^{\alpha}}{\# X} \frac{\#\{Q \in X : Q \subset B_{cRr} \}}{(Rr)^{\alpha}}.$$ Now $\frac{R^{\alpha}}{\# X} \lesssim 1$ since $\gamma^{d}_{\alpha}(X) \sim 1$ and therefore we can conclude that $$\frac{\mu(B_r)}{r^{\alpha}} \lesssim 1 $$ when $r \geq R^{-1}$. On the other hand if $r < R^{-1}$ then $Rr < 1$ and we have $$\mu(B_r) \lesssim R^{-\alpha} (Rr)^d \lesssim R^{-\alpha}(Rr)^{\alpha} = Cr^{\alpha}.$$ It follows that  $$c_{\alpha}(\mu) \lesssim 1$$ and so $\mu$ is $\alpha$-dimensional. As a consequence $$\|Ef\|_{L^2(\mu_R; B^{d}(0,R))} \lesssim_{\delta} R^{s + \delta} c_{\alpha}(\mu_R)^{\frac{1}{2}}\|f\|_{L^2}$$ whenever supp$(f) \subset B^{d-1}(0,2)$. From the definition of $\mu_R$ it follows that $$\|Ef\|_{L^2(X)} \lesssim_{\delta} \big(\frac{\# X}{R^{\alpha}} c_{\alpha}(\mu_R) \big)^{\frac{1}{2}} R^{s+\delta}\|f\|_{L^2}.$$ Finally we choose $r > 0$ and a ball $B_r$ such that $c_{\alpha}(\mu_R) \sim \frac{\mu_R (B_r)}{r^{\alpha}}.$ If $r \geq 1$ we have $$\frac{\# X}{R^{\alpha}} c_{\alpha}(\mu_R) \sim \#X \cdot  \frac{\mu(R^{-1} B_r) }{r^{\alpha}} \lesssim \frac{\#\{Q: Q \subset B_{cr}\} }{r^{\alpha}} \lesssim \gamma^{d}_{\alpha}(X).$$ On the other hand, if $r < 1$ then we can directly estimate $$\frac{\# X}{R^{\alpha}}c_{\alpha}(\mu_R) \sim \# X \cdot \frac{\mu(R^{-1}B_{r})}{r^{\alpha}} \lesssim \frac{|B_{r}|}{r^{\alpha}} \lesssim 1 \lesssim \gamma^{d}_{\alpha}(X).$$ It follows that $$\overline{s_d}(\alpha) \leq s + \delta < s_{d} (\alpha) + \delta.$$ Letting $\delta \rightarrow 0$ completes the proof.

To prove the reverse inequality one uses the fact that $|Ef|$ is essentially constant on cubes of side-length 1 (since $f$ is supported in a ball of radius 2) along with a standard pigeonholing argument to discretize the measure (see for example \cite{DZ}). \end{proof}

\end{document}